\PassOptionsToPackage{colorlinks,citecolor=magenta}{hyperref}

\documentclass[11pt,reqno]{amsart}
\usepackage[margin=1in]{geometry}
\usepackage{amssymb}
\usepackage{colonequals}
\usepackage{nccmath}
\usepackage{mathtools}
\usepackage{syntonly} 
\usepackage{mleftright}
\usepackage{makecell}
\usepackage{rotating}
\usepackage{caption}
\usepackage{lscape}
\usepackage{array}
\usepackage{tabularx,ragged2e,booktabs}
\DeclareMathOperator{\spann}{span}
\DeclareMathOperator{\ad}{ad}

\DeclareMathOperator{\Ric}{Ric}

\DeclareMathOperator{\tr}{tr}

\def \vvv {\quad , \quad}
\newcommand{\aaaa}[1]{\langle#1\rangle}

\usepackage{hyperref}
\usepackage{orcidlink}
\allowdisplaybreaks

\theoremstyle{definition}
\newtheorem{theorem}{Theorem}[section]
\newtheorem{lemma}[theorem]{Lemma}
\newtheorem{prop}[theorem]{Proposition}
\newtheorem{coro}[theorem]{Corollary}

\theoremstyle{definition}
\newtheorem{definition}[theorem]{Definition}
\newtheorem{example}[theorem]{Example}

\newtheorem{remark}[theorem]{Remark}

\numberwithin{equation}{section}

\begin{document}
\newcommand{\spacing}[1]{\renewcommand{\baselinestretch}{#1}\large\normalsize}
\spacing{1.14}

\title[On the Riemann-Finsler Geometry of Tangent Bundle of Lie Groups with\ldots{}]
{On the Riemann-Finsler Geometry of Tangent Bundle of Lie Groups with Two-Dimensional Commutator Subgroup}

\author{Ali Hatami Shahi\orcidlink{0009-0009-7177-7654}}

\address{Department of Pure Mathematics \\ Faculty of Mathematics and Statistics \\ University of \mbox{Isfahan}, Isfahan, Iran. \\ Orcid: 0009-0009-7177-7654} \email{a.hatami@sci.ui.ac.ir \hspace*{20mm} ali.hatami.math.ah@gmail.com}

\author {Hamid Reza Salimi Moghaddam\orcidlink{0000-0001-6112-4259}}

\address{Department of Pure Mathematics\\Faculty of Mathematics and Statistics\\University of \mbox{Isfahan}, Isfahan, Iran. \\ Orcid: 0000-0001-6112-4259}\email{hr.salimi@sci.ui.ac.ir \hspace*{20mm} salimi.moghaddam@gmail.com}



\date{\today}

\begin{abstract}
We begin by studying the Riemannian geometry of the tangent Lie group $TG$ associated with a Lie group $G$ whose commutator subgroup is two-dimensional, equipped with the lift of a left-invariant Riemannian metric on $G$. We establish the relationship between the sectional curvatures of $G$ and those of $TG$. Next, we define a Randers metric on $G$ from a left-invariant Riemannian metric and a left-invariant vector field, and lift it vertically and completely to $TG$. We investigate the conditions under which this Randers metric is of Berwald and Douglas type, respectively, and compute the flag curvatures in the Berwald case. In an addendum, we discuss geodesic vectors and bi-invariant Riemannian metrics on these Lie groups, highlighting the special unimodularity conditions. Finally, we provide explicit formulas for the Riemannian curvature tensor on the tangent bundle of such a Lie group.
\end{abstract}

\maketitle
\section{\bfseries Introduction}\label{s1}
\noindent
In this section, we provide a brief overview of the research history on Lie groups with one- and two-dimensional commutator subgroups, as well as the geometry of the tangent bundle of Riemannian and Finslerian manifolds, particularly the tangent bundle of Lie groups.

Lie groups with one-dimensional commutator subgroups were first introduced by Valerius Kaiser in 1994, in the course of investigating all possible signatures of the Ricci tensor of left-invariant Riemannian metrics on connected n-dimensional Lie groups possessing this property \cite{14}.
He showed that such groups are (up to local isomorphism) isomorphic to a direct product $H\times K$ of Lie groups $H$ and $K$, where $H$ is either one of the $(2m+1)$-dimensional Heisenberg nilpotent groups $H_m$, $m=1,2,\ldots, \left [\frac{n+1}{2}\right]$, consisting of all $(m+2)\times (m+2)$ matrices with 1's along the main diagonal and possibly nonzero entries only in the first row and the last column (see also \cite{10}), or the two-dimensional solvable affine group $A(1)$, consisting of all affine transformations of the real line $\mathbb{R}$; and $K$ is an abelian group.
In fact, he stated that, up to local isomorphism, there exist precisely $\left[\frac{n+1}{2}\right]$ non-isomorphic $n$-dimensional Lie groups with a one-dimensional commutator subgroup.
In 2020, V. A. Le et.al. studied solvable Lie algebras with two-dimensional derived ideals, and classified such Lie algebras by means of matrix theory \cite{17}.
In 2025, the second author of the present article investigated the geometry of left-invariant Riemannian metrics on Lie groups with one- and two-dimensional commutator subgroups and classified all algebraic Ricci solitons in such spaces. He also explicitly computed the Levi-Civita connections, sectional curvatures, and Ricci curvatures of Lie groups with two-dimensional commutator subgroups \cite{19}.

However, the study of the geometry of the tangent bundle of a Riemannian manifold was first initiated in 1958 by S. Sasaki, based on the decomposition of the tangent bundle of a Riemannian manifold into horizontal and vertical subbundles, together with the introduction of a (natural) metric on the tangent bundle \cite{20}.
Subsequently, in 1962, P. Dombrowski introduced the Lie brackets on the tangent bundle of a Riemannian manifold \cite{9}, and later, in 1971, O. Kowalski computed the Levi-Civita connections of Sasaki's metric as well as its Riemann curvature tensor \cite{15}.
Further geometric relations between a Riemannian manifold and its tangent bundle can be found in \cite{1} and \cite{18}.
In 1973, K. Yano and S. Ishihara investigated the geometry of the tangent bundle of a Riemannian manifold from the viewpoint of the vertical and complete subbundles (by means of vertical and complete lifts) \cite{21}.
Following this, numerous other works were devoted to the tangent bundle of a Riemannian manifold, among which one may mention the contributions of S. Gudmundsson and E. Kappos \cite{11}, the dissertation of F. Y. Hindeleh \cite{14}, and another paper by O. Kowalski and M. Sekizawa \cite{16}.
In particular, the Riemannian geometry of the tangent bundle of Lie groups and its relation to the geometry of the underlying Lie group was first examined by F. Asgari and H. R. Salimi Moghaddam in 2017 \cite{2}.
They also, in \cite{3}, analyzed the Riemannian geometry of the tangent bundle of Lie groups with one-dimensional commutator subgroups and presented the geometric quantities associated with this specific family of Lie groups.
Finally, in 2018, the latter authors studied the geometry of the tangent bundle of a Finslerian manifold by considering a left-invariant Randers metric on a base Lie group. They established the flag curvatures of the tangent bundle in the Berwald case, classified all three-dimensional Lie groups whose tangent bundles admit such a metric, and ultimately derived the geodesic vectors of these Lie groups \cite{4}.

The study of lifted Riemannian metrics on the tangent bundle of a manifold $M$ reveals intriguing geometric properties, particularly regarding their curvature when analyzed in the tangent Lie group of a Lie group. Notable contributions to this area can be found in the works of \cite{Peyghan-Seifipour-Blaga} and \cite{Seifipour-Peyghan}, which explore the relationship between the curvature of these lifted metrics and various geometric structures within the tangent bundle. Recently, the investigation has expanded to include the curvature of lifted Finsler metrics on tangent Lie groups, as documented in \cite{Atashafrouz-Najafi-Tayebi-1}, \cite{Atashafrouz-Najafi-Tayebi-2}, and \cite{4}. These studies illustrate how such Finsler metrics are intrinsically linked to their underlying base metrics.
In the present work, we first define a natural left-invariant Riemannian metric on the tangent bundle of a Lie group with a two-dimensional commutator subgroup, in such a way that the vertically and completely lifted vector fields are orthogonal with respect to this metric on the tangent bundle.
We then obtain the Levi-Civita connections, Riemann and sectional curvatures, as well as the Ricci tensor, on the tangent Lie group of these groups, and we establish the relationship between the sectional curvatures on the base Lie group and its tangent Lie group.
Next, we consider a Randers metric determined by a left-invariant Riemannian metric and a left-invariant vector field on G, lift it in the vertical and complete manners to the tangent bundle, and investigate the conditions under which they are of Berwald type. We then compute the flag curvatures in this setting. Finally, we present an example, calculate the aforementioned geometric quantities for it, and, in the appendix, we study the geodesic vectors and bi-invariant Riemannian metrics on these Lie groups, together with the specific conditions for unimodularity of the Lie group. We also provide the formulas for the Riemann curvature tensor on its tangent bundle.

\section{\bfseries Preliminaries}
\label{s2}
In this section, we present the preliminary definitions from Finsler geometry, vertical and complete lifts of vector fields and Riemannian metrics, as well as the necessary propositions to be used in the subsequent sections.

Let $M$ be a real smooth manifold of dimension $n$, and denote by $TM$ its tangent bundle. Given a vector field $X$ on $M$, one obtains a one-parameter family of diffeomorphisms of $TM$ defined by
$\psi_t(y)=y+tX(x)$, for $y\in T_xM$. The infinitesimal generator of this flow, $\psi_t$, is called the vertical lift of $X$ and is denoted by $X^v$.

In contrast, the infinitesimal generator of the (local) one-parameter group of diffeomorphisms $\phi_t$, where $T\phi_t:TM\longrightarrow TM$ is the flow on $M$ generated by $X$, is called the complete lift of $X$ and is denoted by $X^c$. For further discussion of vertical and complete lifts, see, e.g., \cite{13} and \cite{21}. It can be shown that for any two vector fields $X, Y$ we have:
\begin{align*}
& \left[X^c,Y^c\right]= [X,Y]^c, \\
& \left[X^c,Y^v\right]= [X,Y]^v, \\
& \left[X^v,Y^v\right]=0.
\end{align*}
Let $G$ be an $n$-dimensional connected Lie group with multiplication $\mu$, inversion $\iota$, and identity element $e$. For $x\in G$, denote by $l_x$ and $r_x$ the left and right translation maps, respectively.

Identifying $T(G\times G)$ with $TG\times TG$, the tangent map $T\mu$ at $(x, y)$ acts as $T\mu(v, w) = T_y l_x (w) + T_x r_y (v)$,
for $v\in T_x G$ and $w\in T_y G$. This defines a Lie group structure on $TG$, with identity element $0_e\in T_e G$ and inversion $T\iota$ (see, e.g., \cite{12} and \cite{13}).

Moreover, the complete and vertical lifts of any left-invariant vector field on $G$ are left-invariant vector fields on the Lie group $TG$. Indeed, if $\{X_1, ..., X_n\}$ is a basis for the Lie algebra $\frak{g}$ of $G$, then the set $\{X_1^v, ..., X_n^v, X_1^c, ..., X_n^c\}$ forms a basis for the Lie algebra $\tilde{\frak{g}}$ of $TG$.

Similar to the vector fields, one can lift a left-invariant Riemannian metric $g$ from $G$ to a left-invariant Riemannian metric $\widetilde{g}$ on $TG$ as follows (see \cite{2}):
\begin{align*}
& \widetilde{g}\left(X^c, Y^c\right)=g(X,Y), \\
& \widetilde{g}\left(X^v, Y^v\right)=g(X,Y), \\
& \widetilde{g}\left(X^c, Y^v\right)=0,
\end{align*}
where $X,Y\in\frak{g}$. In \cite{2} it is shown that if $\nabla$ and $\widetilde\nabla$ denote the Levi-Civita connections of $g$ and $\widetilde{g}$ respectively then for any two left-invariant vector fields $U$ and $V$ on $G$ we have:
\begin{align}\label{Levi-Civita for tangent general}
 \; & {\widetilde \nabla }_{U^c}V^c={\left( \nabla _UV\right)}^c,  &  {\widetilde \nabla }_{U^v}V^v={\left( \nabla _UV-\frac{1}{2}[U,V]\right)}^c, \\
 \; & {\widetilde \nabla }_{U^c}V^v={\left( \nabla _UV+\frac{1}{2} \ad ^*_VU\right)}^v, & {\widetilde \nabla }_{U^v}V^c={\left( \nabla _UV+\frac{1}{2} \ad ^*_UV\right)}^v.\nonumber
\end{align}

In this article we consider Lie groups whose commutator subgroups have dimension two. Here we give a brief description of them (for more details, see \cite{17} and \cite{19}).

Let $G$  be an $n$-dimensional Lie group with a two-dimensional commutator subgroup $G'$, and let $\mathfrak{g}'$ be the corresponding two-dimensional derived subalgebra of the Lie algebra $\mathfrak{g}$ of $G$. Suppose that $g$ is a left-invariant Riemannian metric on $G$, and $\left\langle \cdot, \cdot \right\rangle$ denotes the inner product induced on $\mathfrak{g}$. We choose the orthonormal basis $\{e_1, e_2\}$ for $\mathfrak{g}'$. Let $\mathcal{P}$ denote the ($n-2$)-dimensional subspace of $\mathfrak{g}$ orthogonal to $\spann\{e_1,e_2\}$ with respect to the $g$.
Then, for some linear maps $\varphi_1, \varphi_2, \psi_1, \psi_2 \colon \mathcal{P} \longrightarrow \mathbb{R}$, we have:
\begin{align*}
& [u,e_1]=\varphi_1\left(u\right)e_1+\varphi _2 (u) e_2, \\
& [u,e_2]= \psi_1 (u) e_1+{\psi }_2 (u) e_2.
\end{align*}
Assume that $a_1,a_2,b_1,b_2\in \mathcal{P}$ are the uniquely determined elements of $\mathcal{P}$ such that for every $u\in \mathcal{P}$,
\begin{align*}
& \varphi _1 (u) =\left\langle a_1,u\right\rangle , \varphi _2 (u) =\left\langle a_2,u\right\rangle,
\\
& {\psi }_1 (u) =\left\langle b_1,u\right\rangle , {\psi }_2 (u) =\left\langle b_2,u\right\rangle.
\shortintertext{Hence,}
& [u,e_1]=\left\langle a_1,u\right\rangle e_1+\left\langle a_2.u\right\rangle e_2,
\\
& [u,e_2]=\left\langle b_1,u\right\rangle e_1+\left\langle b_2,u\right\rangle e_2.
\end{align*}
Suppose that for $u,v\in \mathcal{P}$ we have $ [u,v] =B_1 (u,v) e_1+B_2 (u,v) e_2$ where $B_1 (u,v) $ and $B_2 (u,v) $ are skew-symmetric bilinear forms on the subspace $\mathcal{P}$. Then there exist uniquely skew-symmetric linear maps $f_1,f_2: \mathcal{P}\longrightarrow \mathcal{P}$ such that for all $u,v\in \mathcal{P}$,
\begin{align*}
& B_1 (u,v) =\left\langle f_1 (u) ,v\right\rangle, \quad B_2 (u,v) =\left\langle f_2 (u) ,v\right\rangle.
\shortintertext{Therefore,}
&[u,v] =\left\langle f_1 (u) ,v\right\rangle e_1+\left\langle f_2 (u) ,v\right\rangle e_2, \qquad \forall u,v\in \mathcal{P}.
\end{align*}
Using the Jacobi identity and the Cauchy-Schwarz inequality, we can show that $b_1=\lambda a_2$, for some $\lambda\in\Bbb{R}$. But we do not use this equation to show the symmetry of the formulas.
\begin{prop}\label{t2.22}
\cite{19}.
For an $n$-dimensional Lie group with a two-dimensional commutator subgroup equipped with a left-invariant metric $g$, the Levi-Civita connection and the sectional curvatures are given by:
\begin{align*}
1) \; &
 \nabla _{e_1}e_1=a_1 \vvv \nabla _{e_1}e_2=\frac{1}{2}\left(a_2+b_1\right) \vvv \nabla _{e_1}u=-\left(\langle a_1,u\rangle e_1+\frac{1}{2}\left(\langle b_1+a_2,u\rangle e_2+f_1 (u) \right)\right) & \\
& \nabla _{e_2}e_1=\frac{1}{2}\left(a_2+b_1\right) \vvv \nabla _{e_2}e_2=b_2 \vvv \nabla _{e_2}u=-\left(\langle b_2,u\rangle e_2+\frac{1}{2}\left(\langle b_1+a_2,u\rangle e_1+f_2 (u) \right)\right) & \\
& \nabla _ve_1=\frac{1}{2}\left(\langle a_2-b_1,v\rangle e_2-f_1(v)\right) \vvv\nabla _ve_2=\frac{1}{2}\left(\langle b_1-a_2,v\rangle e_1-f_2(v)\right) & \\
& \nabla _vu=\frac{1}{2}\left(\langle f_1(v),u\rangle e_1+\langle f_2(v),u\rangle e_2\right) &
\end{align*}
where $u,v\in \mathcal{P}$, and
\begin{flalign*}
2) \; & K\left(e_1,e_2\right)=\frac{1}{4}\| a_2+b_1\|^2-\langle a_1,b_2\rangle, \\
& K\left(e_1,u\right)=\frac{1}{4}\left({\langle b_1,u\rangle }^2-3{\langle a_2,u\rangle }^2+{\| f_1 (u) \| }^2\right)-{\langle a_1,u\rangle }^2-\frac{1}{2}\langle a_2,u\rangle \langle b_1,u\rangle, & \\
& K\left(e_2,u\right)=\frac{1}{4}\left({\langle a_2,u\rangle }^2-3{\langle b_1,u\rangle }^2+{\| f_2 (u) \|}^2\right)-{\langle b_2,u\rangle }^2-\frac{1}{2}\langle a_2,u\rangle \langle b_1,u\rangle, & \\
& K (u,v) = -\frac{3}{4}\left({\langle u,f_1(v)\rangle }^2+{\langle u,f_2(v)\rangle }^2\right), &
\end{flalign*}
where $u,v\in \mathcal{P}$ are arbitrary orthonormal vectors and ${\left\|\cdot \right\|}^2=\left\langle \cdot , \cdot \right\rangle $.
\end{prop}
We now present the preliminaries of Finsler geometry.
\begin{definition}\label{t2.24}
\cite{6,7}.
A Randers metric on a manifold $M$ is a Finsler metric of the following form, with the condition ${\left\|\beta \right\|}_{\alpha }<1$ (equivalently, ${\left\|X\right\|}_{\alpha }<1$):
\[F_x\left(y\right)=\alpha (x,y)+\beta (x,y) =\alpha (x,y)+g\left(X_x,y\right),\]
where $\alpha (x,y) =\sqrt{g_x(y,y)}$ is the Riemannian norm and $\beta =b_idx^i$ is a 1-form such that $\beta (x,y) =b_iy^i$. Also, $X$ denotes the vector field on $M$ corresponding to the 1--form $\beta $.
\end{definition}
\begin{definition}\label{t2.25}
\cite{6}.
A Finsler metric $F$ on a smooth manifold $M$ is said to be of Berwald type if, in the standard local coordinates $\left(x^i,y^i\right)$ on $TM\setminus \{0\}$, the Christoffel symbols $\Gamma^i_{jk}$ of the Chern connection depend only on $x\in M$ (i.e., are independent of $y$).
\end{definition}
\begin{lemma}\label{t2.26}
\cite{7}.
Suppose $F_x\left(y\right)=\alpha (x,y)+\beta (x,y) =\alpha (x,y) +g\left(X_x,y\right)$ is a Randers metric. Then the following statements are equivalent:
\begin{enumerate}
\item
$F$ is of Berwald type.
\item
$\beta $ is parallel with respect to $\alpha $ (equivalently, the vector field $X$ is parallel with respect to the Riemannian metric $\alpha $).
\end{enumerate}
\end{lemma}

S. Bacso and M. Matsumoto, as a generalization of Berwald-type metrics from the viewpoint of geodesic equations, considered Douglas-type metrics (for the definition see \cite{5}).
\begin{definition}\label{t2.27}
A Randers metric $F=\alpha +\beta $ is said to be of Douglas type if the 1--form $\beta $ is closed. Equivalently, $d\beta =0$.
\end{definition}
S.Deng, however, in his book \cite{7}, studied the Douglas property of Randers metrics and subsequently provided a practical necessary and sufficient condition for such Finsler metrics to be of Douglas type on homogeneous manifolds. The necessary and sufficient condition for the special case of Lie groups is given below.
\begin{prop}\label{t2.28}
A Randers metric $F$ is of Douglas type if and only if, for all $Y,Z\in \mathfrak{g}$, $g\left(X, [Y,Z] \right)=0$
\end{prop}
\section{\bfseries The Riemann-Finsler Geometry of Tangent Lie Groups Having Two-Dimensional Commutator Subgroups}
\label{s.3}
In this section, we study and analyze the Riemann-Finsler geometry of tangent Lie groups associated with the family of Lie groups having two-dimensional commutator subgroups, and we compute some of their related geometric quantities.

Let $G$ be an $n$-dimensional Lie group with Lie algebra $\mathfrak{g}$ and a two-dimensional derived subalgebra ${\mathfrak{g}}'=\spann\left\{e_1,e_2\right\}$, and let $g$ be an arbitrary left-invariant Riemannian metric on $G$ such that $e_1$ and $e_2$ are orthonormal with respect to $g$. Also, let $\mathcal{P}=\spann\left\{u_1,u_2,\ldots ,u_{n-2}\right\}$ be an ($n-2$)-dimensional subspace of $\mathfrak{g}$, generated by orthonormal vectors ${\left\{u_i\right\}}^{n-2}_{i=1}$ with respect to the same metric, such that ${\mathfrak{g}}'{\bot }_g\mathcal{P}$ and $\mathfrak{g}={\mathfrak{g}}'\oplus \mathcal{P}$. Using complete and vertical lifts of vector fields, we construct a natural Riemannian metric on the tangent Lie group $TG$, which is clearly left-invariant. Let $\widetilde \nabla $ denote the Levi-Civita connection on the tangent bundle $TG$.

In the first step, we have the following lemma, which provides explicit formulas for these Levi-Civita connections on the tangent bundle of such Lie groups.
\begin{lemma}\label{t3.1}
Let $G$ be a Lie group with a two-dimensional commutator subalgebra, and $g$ a left-invariant Riemannian metric on $G$.
Let $\tilde{g}$ be the natural lifted left-invariant metric on $TG$, and let $\widetilde{\nabla}$ be the Levi-Civita connection with respect to $\tilde{g}$. Then, on $TG$ we have (see the table~\ref{j.1}), where $u$ and $v$ are arbitrary elements of the ($n-2$)-dimensional subspace $\mathcal{P}$ of the Lie algebra $\mathfrak{g}$.
\begin{table}[!htb]
\caption{The Levi-Civita connection of the tangent Lie group}
{{\centering\label{j.1}%
\setlength{\arraycolsep}{2pt}%
\resizebox{\textwidth}{!}%
{\renewcommand{\arraystretch}{2}%
$\begin{array}{|c|c|c|c|c|c|c|} \hline
 \tilde\nabla & e^c_1 & e^v_1 & e^c_2 & e^v_2 & u^c & u^v
 \\ \hline
 e^c_1 & a^c_1 & \frac{1}{2}a^v_1 & \frac{1}{2}\left(a^c_2+b^c_1\right) & \frac{1}{2}a^v_2 &
\makecell{-\big(\aaaa{a_1,u}e_1^c+\\\frac{1}{2}\aaaa{b_1+a_2,u}e_2^c+\\\frac{1}{2}(f_1(u))^c\big)} & \makecell{-\frac{1}{2} \big(\aaaa{a_1,u} e_1^v +\\ \aaaa{a_2,u} e_2^v\big)}
 \\ \hline
 e^v_1 & \frac{1}{2}a^v_1 & a^c_1 & \frac{1}{2}b^v_1 & \frac{1}{2}\left(a^c_2+b^c_1\right) & \makecell{-\big(\aaaa{a_1,u}e_1^v+\\\frac{1}{2}\aaaa{b_1+a_2,u} e_2^v + \\\frac{1}{2} (f_1(u))^v\big)} & \makecell{-\frac{1}{2} \big(\aaaa{a_1,u}e_1^c+\\\aaaa{b_1,u}e_2^c+(f_1(u))^c\big)}
 \\ \hline
 e^c_2 & \frac{1}{2}\left(a^c_2+b^c_1\right) & \frac{1}{2}b^v_1 & b^c_2 & \frac{1}{2}b^v_2 & \makecell{-\big(\frac{1}{2}\aaaa{b_1+a_2,u}e_1^c+\\\aaaa{b_2,u}e_2^c+\\\frac{1}{2}(f_2(u))^c\big)} & \makecell{-\frac{1}{2} \big(\aaaa{b_1,u}e_1^v+\\\aaaa{b_2,u}e_2^v\big)}
 \\ \hline
 e^v_2 & \frac{1}{2}a^v_2 & \frac{1}{2}\left(a^c_2+b^c_1\right) & \frac{1}{2}b^v_2 & b^c_2 & \makecell{-\big(\frac{1}{2}\aaaa{b_1+a_2,u}e_1^v+\\\aaaa{b_2,u}e_2^v+\\\frac{1}{2}(f_2(u))^v\big)} & \makecell{-\frac{1}{2} \big(\aaaa{a_2,u}e_1^c+\\\aaaa{b_2,u}e_2^c+\\(f_2(u))^c\big)}
 \\ \hline
v^c &  \makecell{\frac{1}{2}\big(\aaaa{a_2-b_1,v}e_2^c\\-(f_1(v))^c\big)} & \makecell{\frac{1}{2} \big(\aaaa{a_2-b_1,v}e_2^v\\-(f_1(v))^v\big)} & \makecell{\frac{1}{2} \big(\aaaa{b_1-a_2,v}e_1^c\\-(f_2(v))^c\big)} &
\makecell{\frac{1}{2} \big(\aaaa{b_1-a_2,v} e_1^v-\\(f_2(v))^v\big)} & \makecell{\frac{1}{2} \big(\aaaa{a_2,v}e_1^c \\+\aaaa{b_2,v}e_2^c +\\(f_2(v))^c\big)} & \makecell{\frac{1}{2} \big(\aaaa{f_1(v),u}e_1^v + \\\aaaa{f_2(v),u}e_2^v\big)}
 \\ \hline
v^v & \makecell{\frac{1}{2} \big(\aaaa{a_1,v}e_1^v+\\\aaaa{a_2,v}e_2^v\big)} & \makecell{-\frac{1}{2} \big(\aaaa{a_1,v}e_1^c+\\\aaaa{b_1,v}e_2^c+(f_1(v))^c\big)} & \makecell{\frac{1}{2} \big( \aaaa{b_1,v}e_1^v+\\\aaaa{b_2,v}e_2^v\big)} & \makecell{-\frac{1}{2} \big( \aaaa{a_2,v}e_1^c+\\\aaaa{b_2,v}e_2^c+(f_2(v))^c\big)} & \makecell{\frac{1}{2} \big(\aaaa{f_1(v),u}e_1^v+\\\aaaa{f_2(v),u}e_2^v\big)} & 0
\\ \hline
\end{array}$}
}}
\end{table}
\noindent
\end{lemma}

\begin{proof}
It suffices to observe that:
\begin{align*}
& g\left( \ad ^*_{e_1}e_1, c_1e_1+c_2e_2\right) = g\big(e_1, \ad _{e_1}\left(c_1e_1+c_2e_2\right)\big) = g\big(e_1, c_1\left[e_1,e_1\right]+c_2\left[e_1,e_2\right]\big) = 0,\\
& g\left( \ad ^*_{e_1}e_1, u\right)=g\left(e_1, \ad _{e_1}u\right) = g\big(e_1,\left[e_1,u\right]\big) = g\big(-\left\langle a_1,u\right\rangle e_1-\left\langle a_2,u\right\rangle e_2,e_1\big) = \left\langle -a_1,u\right\rangle.
\end{align*}
Hence, $ \ad ^*_{e_1}e_1=a_1$.
Therefore, by equations \eqref{Levi-Civita for tangent general}, we have:
\[{\widetilde \nabla }_{e^c_1}e^v_1={\left( \nabla _{e_1}e_1+\frac{1}{2} \ad ^*_{e_1}e_1\right)}^v={\left(a_1+\frac{1}{2}\left(-a_1\right)\right)}^v=\frac{1}{2}a^v_1.\]
Similarly, we compute:
\[ \ad ^*_{e_1}e_2=-a_2 , \ad ^*_{e_2}e_1=-b_1 , \ad ^*_{e_2}e_2=-b_2.\]
Moreover,
\[\!\!\left. \begin{array}{@{}l@{}}
g(\ad_{e_1}u, c_1e_1+c_2e_2) = g\big(u, \ad _{e_1}(c_1e_1+c_2e_2)\big) = g\big(u,c_1[e_1,e_1 ]+c_2[e_1,e_2]\big)=0 \\
g\left( \ad ^*_{e_1}u , v\right)=g\left(u, \ad _{e_1}v\right)=g\left(u,\left[e_1,v\right]\right)=g\big(u,-\left\langle a_1,v\right\rangle e_1-\left\langle a_2,v\right\rangle e_2\big)=0
\end{array}\right\}\!
\Rightarrow \!\ad^*_{e_1}\!u=0.
\]
By similar arguments, $ \ad ^*_{e_2}u=0$.

For computing $ \ad ^*_ue_1\in \mathfrak{g}$, suppose $ \ad ^*_ue_1={\xi }_1e_1+{\xi }_2e_2+\sum^n_{i=3}{{\xi }_iu_i}$ where ${\left\{u_i\right\}}^n_{i=3}$ is an orthonormal basis of the ($n-2$)-dimensional subspace $\mathcal{P}$ and ${\xi }_i$s are scalers. Direct computation gives:
\begin{align*}
& {\xi }_1=g\left( \ad ^*_ue_1,e_1\right)=g\left(e_1, \ad _ue_1\right)=g\left(e_1,\left\langle a_1,u\right\rangle e_1+\left\langle a_2,u\right\rangle e_2\right)=\left\langle a_1,u\right\rangle,\\
& {\xi }_2=g\left( \ad ^*_ue_1,e_2\right)=g\left(e_1, \ad _ue_2\right)=g\left(e_1, \left\langle b_1,u\right\rangle e_1+\left\langle b_2,u\right\rangle e_2\right)=\left\langle b_1,u\right\rangle, \\
& {\xi }_i=g\left( \ad ^*_ue_1,u_i\right)=g\left(e_1, \ad _uu_i\right)=g\left(e_1, \left\langle f_1 (u) ,u_i\right\rangle e_1+\left\langle f_2 (u) ,u_i\right\rangle e_2\right)=\left\langle f_1 (u) ,u\right\rangle.
\end{align*}
So we have:
\begin{align*}
\ad ^*_ue_1=\left\langle a_1,u\right\rangle e_1+\left\langle b_1,u\right\rangle e_2+\sum^n_{i=3}{\left\langle f_1 (u) ,u_i\right\rangle u_i}=\left\langle a_1,u\right\rangle e_1+\left\langle b_1,u\right\rangle e_2+f_1 (u).
\end{align*}
Similarly, $ \ad ^*_ue_2=\left\langle a_2,u\right\rangle e_1+\left\langle b_2,u\right\rangle e_2+f_2(u)$.
Also, one can check that
\begin{align*}
& \begin{aligned}
 g\left( \ad ^*_uv, c_1e_1+c_2e_2\right)
 &=g\left(v, c_1 \ad _ue_1+c_2 \ad _ue_2\right)\\
 &=g\left(v,\left(c_1\left\langle a_1,u\right\rangle +c_2\left\langle b_1,u\right\rangle \right)e_1+\left(c_1\left\langle a_2,u\right\rangle +c_2\left\langle b_2,u\right\rangle \right)e_2\right)=0,
\end{aligned}\\
& \begin{aligned}
 g\left( \ad ^*_uv,w\right)
 &=g\left(v, \ad _uw\right)=g\left(v,\left\langle f_1 (u) ,v\right\rangle e_1+\left\langle f_2 (u) ,v\right\rangle e_2\right)=0.
\end{aligned}
\end{align*}
So, $\ad ^*_uv=0$.
By a similar way, $\ad ^*_vu=0$. Also, using the same method as above, we obtain:
\[ \ad ^*_{e_1}u= \ad ^*_{e_2}u=0,\]
where $u,v,w\in \mathcal{P}$.

Now, equations \eqref{Levi-Civita for tangent general} together with the Proposition~\ref{t2.22}, complete the proof.
\end{proof}
\begin{theorem}\label{t3.2}
Let $G$ be a Lie group with a two-dimensional derived subalgebra. Then the sectional curvatures are:
\begin{align*}
&\tilde{K}\left(e^c_1,e^c_2\right)=\tilde{K}\left(e^v_1,e^v_2\right)=\frac{1}{4}{\left\|a_2+b_1\right\|}^2-\left\langle a_1,b_2\right\rangle \vvv \tilde{K}\left(e^c_1,e^v_2\right)=\frac{1}{4}{\left\|a_2\right\|}^2-\left\langle a_1,b_2\right\rangle, \\
& \tilde{K}\left(e^c_1,e^v_1\right)=-\frac{3}{4}{\left\|a_1\right\|}^2 \vvv \tilde{K}\left(e^v_1,e^c_2\right)=\frac{1}{4}{\left\|b_1\right\|}^2-\left\langle a_1,b_2\right\rangle, \\
& \tilde{K}\left(e^c_2,e^v_2\right)=-\frac{3}{4}{\left\|b_2\right\|}^2 \vvv \tilde{K}\left(u^v,v^v\right)=\tilde{K}\left(u^c,u^v\right)=0,\\
& \tilde{K}\left(u^c,v^c\right)=\tilde{K}\left(u^c,v^v\right)=-\frac{3}{4}\left({\left\langle f_1 (u) ,v\right\rangle }^2+{\left\langle f_2 (u) ,v\right\rangle }^2\right), \\
& \tilde{K}\left(u^c,e^c_1\right)=\tilde{K}\left(u^c,e^v_1\right)=\frac{1}{4}{\left\|f_1 (u) \right\|}^2-{\left\langle a_1,u\right\rangle }^2+\frac{1}{4}\left\langle u,a_2+b_1\right\rangle \left\langle u,b_1-3a_2\right\rangle, \\
& \tilde{K}\left(u^c,e^c_2\right)=\tilde{K}\left(u^c,e^v_2\right)=\frac{1}{4}{\left\|f_2 (u) \right\|}^2-{\left\langle b_2,u\right\rangle }^2+\frac{1}{4}\left\langle u,a_2+b_1\right\rangle \left\langle u,a_2-3b_1\right\rangle, \\
& \tilde{K}\left(u^v,e^c_1\right)=-\frac{3}{4}\left({\left\langle a_1,u\right\rangle }^2+{\left\langle a_2,u\right\rangle }^2\right) \vvv \tilde{K}\left(u^v,e^v_1\right)=\frac{1}{4}\left({\left\|f_1 (u) \right\|}^2+{\left\langle a_1,u\right\rangle }^2+{\left\langle b_1,u\right\rangle }^2\right), \\
& \tilde{K}\left(u^v,e^c_2\right)=-\frac{3}{4}\left({\left\langle b_1,u\right\rangle }^2+{\left\langle b_2,u\right\rangle }^2\right) \vvv \tilde{K}\left(u^v,e^v_2\right)=\frac{1}{4}\left({\left\|f_2 (u) \right\|}^2+{\left\langle a_2,u\right\rangle }^2+{\left\langle b_2,u\right\rangle }^2\right),
\end{align*}
where $\{u,v\}$ is an arbitrary orthonormal subset of the ($n-2$)-dimensional subspace $\mathcal{P}$.
\end{theorem}
\begin{proof}
The proof is a direct computation using the Riemann curvature tensor given in Lemma \ref{t4.4} of the Addendum. For instance, we see that:
\begin{align*}
 \tilde{K}\left(e^c_1,e^v_1\right)
 &= \frac{\tilde{g} \left(\tilde{R}\left(e^c_1,e^v_1\right)e^v_1 , e^c_1\right)}{\tilde{g}\left(e^c_1,e^c_1\right)\tilde{g}\left(e^v_1,e^v_1\right)-{\tilde{g}}^2\left(e^c_1,e^v_1\right)}
\\
&=\tilde{g}\left(-\frac{1}{4}\left(3{\left\|a_1\right\|}^2e^c_1+\left\langle a_1,b_1+2a_2\right\rangle e^c_2+(f_1\left(a_1\right))^c\right), e^c_1\right)=-\frac{3}{4}{\left\|a_1\right\|}^2.
\end{align*}
The proofs of the remaining formulas are similar.
\end{proof}

The next corollary follows from the combination of the preceding theorem and the second assertion of Proposition ~\ref{t2.22}.

\begin{coro}\label{t3.3}
(1)
The tangent bundle of any Lie group with a two-dimensional derived subalgebra, at each point, has positive, negative, and zero sectional curvatures.\\
(2)
 The relationship between the sectional curvatures on the tangent Lie group and the base Lie group is given by:
\begin{enumerate}
\item[I)]
$ \tilde{K}\left(u^c, e^c_i\right)=\tilde{K}\left(u^c,e^v_i\right)=K\left(u,e_i\right) , i=1,2,$
\item[II)] $\tilde{K}\left(e^c_1,e^c_2\right)=\tilde{K}\left(e^v_1,e^v_2\right)=K(e_1,e_2),$
\item[III)] $\tilde{K}\left(u^c,v^c\right)=\tilde{K}\left(u^c,v^v\right)=K (u,v), $
\end{enumerate}
where $\{u,v\}$ is an orthonormal subset in $\mathcal{P}$.
\end{coro}
\begin{theorem}\label{t3.4}
Let $G$ be an $n$-dimensional Lie group with a two-dimensional commutator subgroup, and let $\{u_i\}^{n-2}_{i=1}$ be an orthonormal basis of the $(n-2)$-dimensional subspace $\mathcal{P}$, so that $\{e_1,e_2,u_1,u_2,\ldots ,u_{n-2}\}$ is an orthonormal basis of $\mathfrak{g}$.
Then the Ricci curvatures on the tangent Lie group are:
\begin{enumerate}
\item[1)]
$\tilde{r}\left(e^c_1\right)=\widetilde{\Ric }\left(e^c_1,e^c_1\right)=r\left(e_1\right)-\frac{1}{2}\left(3{\left\|a_1\right\|}^2+{\left\|a_2\right\|}^2\right)-\left\langle a_1,b_2\right\rangle, $
\item[2)]
$\tilde{r}\left(e^c_2\right)=\widetilde{\Ric }\left(e^c_2,e^c_2\right)=r\left(e_2\right)-\frac{1}{2}\left(3{\left\|b_2\right\|}^2+{\left\|b_1\right\|}^2\right)-\left\langle a_1,b_2\right\rangle, $
\item[3)]
$\tilde{r}\left(e^v_1\right)=\widetilde{\Ric }\left(e^v_1,e^v_1\right)=2r\left(e_1\right)+\frac{1}{2}\left({\left\|a_1\right\|}^2+{\left\|a_2\right\|}^2\right),$
\item[4)]
$\tilde{r}\left(e^v_2\right)=\widetilde{\Ric }\left(e^v_2,e^v_2\right)=2r\left(e_2\right)+\frac{1}{2}\left({\left\|b_1\right\|}^2+{\left\|b_2\right\|}^2\right),$
\item[5)]
$5)\tilde{r}\left(u^c\right)=\widetilde{\Ric }\left(u^c,u^c\right)=2r (u),$
\item[6)]
$\tilde{r}\left(u^v\right)=\widetilde{\Ric }\left(u^v,u^v\right)=r (u) +\frac{1}{2}\left({\left\langle a_1,u\right\rangle }^2+{\left\langle b_2,u\right\rangle }^2+2\left\langle a_2,u\right\rangle \left\langle b_1,u\right\rangle \right),$
\item[7)]
$\widetilde{\Ric }\left(e^c_1,e^c_2\right)=\Ric \left(e_1,e_2\right)-\frac{1}{2}\left(\left\langle a_1,a_2+2b_1\right\rangle +\left\langle b_2,b_1+2a_2\right\rangle \right),$
\item[8)]
$\widetilde{\Ric }\left(e^v_1,e^v_2\right)=2Ric\left(e_1,e_2\right)+\frac{1}{2}\left(\left\langle a_1,b_1\right\rangle +\left\langle a_2,b_2\right\rangle \right),$
\item[9)]
$\widetilde{\Ric }\left(u^c,v^c\right)=2\Ric (u,v),$
\item[10)]
$\widetilde{\Ric }\left(u^v,v^v\right)=\Ric (u,v) +\frac{1}{2}\Big(\langle a_1,u\rangle \langle a_1,v\rangle +\left\langle b_2,u\right\rangle \left\langle b_2,v\right\rangle +\left\langle b_1,u\right\rangle \left\langle a_2,v\right\rangle +\left\langle a_2,u\right\rangle \left\langle b_1,v\right\rangle \Big),$
\item[11)]
$\widetilde{\Ric }\left(e^c_1,u^c\right)=2Ric(e_1,u),$
 \item[12)]
$\widetilde{\Ric }\left(e^c_2,u^c\right)=2Ric(e_2,u),$
\item[13)]
$\widetilde{\Ric }\left(e^v_1,u^v\right)=\Ric \left(e_1,u\right)-\frac{1}{2}\left\langle u,f_1\left(a_1+b_2\right)\right\rangle,$
\item[14)]
$\widetilde{\Ric }\left(e^v_2,u^v\right)=\Ric \left(e_2,u\right)-\frac{1}{2}\left\langle u,f_2\left(a_1+b_2\right)\right\rangle,$
\item[15)]
$\begin{aligned}[t]
 \widetilde{\Ric }\left(e^c_1,e^v_2\right)
&=\widetilde{\Ric }\left(e^c_1,e^v_1\right)
=\widetilde{\Ric }\left(e^c_2,e^v_2\right)
=\widetilde{\Ric }\left(e^v_1,e^c_2\right)
=\widetilde{\Ric }\left(u^c,v^v\right)
=\widetilde{\Ric }\left(u^c,u^v\right)\\
& =\widetilde{\Ric }\left(e^v_1,u^c\right)
 =\widetilde{\Ric }\left(e^v_2,u^c\right)
 =\widetilde{\Ric }\left(e^c_1,u^v\right)
 =\widetilde{\Ric }\left(e^c_2,u^v\right)=0.
\end{aligned}$
\end{enumerate}
\end{theorem}
\begin{proof}
Easily, we can see that the Ricci curvatures on the tangent bundle $TG$ are defined as follows:
\[\widetilde{\Ric}\left(X^k_i,X^l_j\right) = \sum^n_{m=1} \tilde{g}\left(\tilde{R}\left(X^c_m,X^k_i\right)X^l_j , X^c_m\right) +\sum^n_{m=1} \tilde{g}\left(\tilde{R}\left(X^v_m,X^k_i\right)X^l_j , X^v_m\right), \]
where $k,l\in \{c,v\}$, $i,j\in \{1,2,\ldots ,n\}$, and $\{X_1, X_2, \ldots , X_n\}$ is an orthonormal basis consisting of left-invariant vector fields for the Lie algebra $\mathfrak{g}$ of the Lie group $G$. For instance, we compute case 11. The proof of the rest is similar.
\begin{align*}
\widetilde{\Ric}(e^c_1,u^c)
\!& =\!\tilde{g}\left(\tilde{R}(e^c_1,e^c_1)u^c,e^c_1\right)+\tilde{g}\left(\tilde{R}(e^c_2,e^c_1)u^c,e^c_2\right) + \tilde{g}\left(\tilde{R}(e^v_1,e^c_1)u^c,e^v_1\right)
\\
& +\tilde{g}\left(\tilde{R}\left(e^v_2,e^c_1\right)u^c,e^v_2\right)+\sum^{n-2}_{i=1}{\tilde{g}(\tilde{R}\left(u^c_i,e^c_1\right)u^c,u^c_i)}
+\sum^{n-2}_{i=1}{\tilde{g}\left(\tilde{R}\left(u^v_i,e^c_1\right)u^c,u^v_i\right)}
\\
& =\tilde{g}\left(\tilde{R}\left(u^c,e^c_2\right)e^c_2,e^c_1\right)
+\tilde{g}\left(\tilde{R}\left(u^c,e^v_1\right)e^v_1,e^c_1\right)
+\tilde{g}\left(\tilde{R}\left(u^c,e^v_2\right)e^v_2,e^c_1\right)
\\
& +\sum^{n-2}_{i=1}{\tilde{g}(\tilde{R}\left(u^c,u^c_i\right)u^c_i,e^c_1})+\sum^{n-2}_{i=1}{\tilde{g}\left(\tilde{R}\left(u^c,u^v_i\right)u^v_i,e^c_1\right)}
\\
& =\frac{1}{4}\left\langle u,f_2\left(b_1+a_2\right)-2f_1\left(b_2\right)\right\rangle -\frac{1}{4}\left\langle u,f_1\left(a_1\right)\right\rangle +\frac{1}{4}\left\langle u,f_2\left(a_2\right)-2f_1\left(b_2\right)\right\rangle
\\
& +\frac{1}{4}\sum^{n-2}_{i=1}{\left\langle f_2 (u) ,u_i\right\rangle \left\langle b_1+3a_2,u_i\right\rangle }
 +\sum^{n-2}_{i=1}{\left\langle f_1 (u) ,u_i\right\rangle \left\langle a_1,u_i\right\rangle }
+\frac{3}{4}\sum^{n-2}_{i=1}{\left\langle f_1 (u) ,u_i\right\rangle \left\langle a_1,u_i\right\rangle }
\\
& +\frac{3}{4}\sum^{n-2}_{i=1}{\left\langle f_2 (u) ,u_i\right\rangle \left\langle a_2,u_i\right\rangle }
\\
& =\frac{1}{4}\left\langle u , f_2(b_1)+f_2(a_2)-2f_1(b_2)-f_1(a_1)+f_2(a_2)-2f_1(b_2)\right\rangle +\frac{1}{4}\left\langle f_2 (u) ,b_1+3a_2\right\rangle
\\
& +\left\langle f_1 (u) ,a_1\right\rangle +\frac{3}{4}\left\langle f_1 (u) ,a_1\right\rangle +\frac{3}{4}\left\langle f_2 (u) ,a_2\right\rangle
\\
&=\frac{1}{4}\left\langle u , f_2(b_1)-f_1(a_1)+2f_2(a_2)-4f_1(b_2)\right\rangle \!-\!\frac{7}{4}\left\langle u , f_1(a_1)\right\rangle
\!-\!\frac{1}{4}\left\langle u , f_2(b_1)+6f_2(a_2)\right\rangle
\\
&-\left\langle u , f_2(a_2)+f_1(b_2+2a_1)\right\rangle =2\left\langle u,-f_1(a_1)\right\rangle -\left(\left\langle u,f_2(a_2)\right\rangle +\left\langle u,f_1(b_2)\right\rangle \right)
\\
&=2\Ric (e_1,u). \qedhere
\end{align*}
\end{proof}

\begin{theorem}\label{t3.5}
Let $G$ be a Lie group with a two-dimensional commutator subgroup equipped with a left-invariant Randers metric $F$ defined by a left-invariant Riemannian metric $g$ and a left-invariant vector field $X\neq 0$ such that $\|X\|_g<1$, where $\|\cdot\|_g=\sqrt{g(\cdot ,\cdot)}$.
Then:
\begin{enumerate}
\item[1.]
$F$ is of Douglas type if and only if $X\in \mathcal{P}$.
\item[2.]
$F$ is of Berwald type if and only if $X\in \ker(f_1)\cap \ker (f_2)\cap \spann \{a_1,b_2,a_2+b_1\}^\bot$.
\end{enumerate}
\end{theorem}
\begin{proof}
(1)
According to Proposition~\ref{t2.28}, it is clear.
\\
(2)
 In \cite{8}, it is proved by Lemma~\ref{t2.26} and Koszul's formula that $F$ is of Berwald type if and only if the following conditions hold:
\begin{enumerate}
\item[i)]
$g\big([X,y] ,z\big)+g\big([X,z],y\big)=0,$
\item[ii)]
$g\big(X, [y,z] \big)=0, \qquad \forall y,z\in \mathfrak{g}.$
\end{enumerate}
Now, to check condition (i), we consider the following various cases:
\begin{enumerate}
\item[(a)]
If $y,z\in \mathcal{P}$, then by Koszul's formula and the fact that $X$ is parallel with respect to the Riemannian connection induced by $g$, condition (i) holds.
\item[(b)]
If $y=k_1e_1+k_2e_2$ and $z=l_1e_1+l_2e_2$, then:
\[ g([X,y] ,z) +g([X,z],y)=0 \Longleftrightarrow 2k_1l_1\langle a_1,X\rangle +(k_2l_1+k_1l_2)(\langle a_2,X\rangle + \langle b_1,X \rangle )+2k_2l_2 \langle b_2,X\rangle =0.\]
\item[(c)]
If $y\in \mathcal{P}$ and $z=k_1e_1+k_2e_2$, we have:
\[g\left( [X,y] ,z\right) + g\left([X,z],y\right) = 0 \Longleftrightarrow k_1 \langle f_1 (X) ,y \rangle +k_2 \langle f_2 (X) ,y \rangle = 0, \qquad \forall y\in \mathfrak{g}. \]
\end{enumerate}
On the other hand, to check condition (2), since every Berwald-type Randers metric $F$ is also Douglas (see \cite{7}), the first part of the proposition yields $X\in \mathcal{P}$.
Therefore, considering the cases above, the statement is satisfied.
\end{proof}

\begin{remark}\label{t3.6}
(1)
In the special case $n=3$, it is easily seen that $G$ does not admit any Berwald-type Randers metric.
\\
(2)
In the case $n=4$, a necessary condition for F to be Berwald is that
\[X\in \spann \{a_1,b_2,a_2+b_1\}^\bot.\]
\end{remark}
Suppose $F$ is a left-invariant Randers metric on $G$ as defined in the above theorem on the Lie group $G$.
Its complete and vertical lifts on the tangent Lie group $TG$ are defined as follows (see \cite{4}):
\[
F^c\left( (x,y) ,\tilde{Z}\right)=\sqrt{\tilde{g}(\tilde{Z},\tilde{Z})} + \tilde{g}\left(X^c (x,y) ,\tilde{Z}\right), \ \
F^v\left( (x,y) ,\tilde{Z}\right) = \sqrt{\tilde{g}(\tilde{Z},\tilde{Z})}+\tilde{g}\left(X^v (x,y) ,\tilde{Z}\right),
\]
where $x\in G$, $y\in T_xG$, and $\tilde{Z}\in T_{(x,y)}TG$.

\begin{coro}\label{t3.7}
Let $F^c$ and $F^v$ be the complete and vertical lifts of the left-invariant Randers metric $F$ defined above. Then
\begin{enumerate}
\item[1.]
$F^c$ is a Berwald-type Randers metric if and only if
\[X\in \ker(f_1)\cap \ker(f_2)\cap \spann\{a_1,b_2,a_2+b_1\}^\bot\]
\item[2.]
$F^v$ is a Berwald-type Randers metric if and only if $X\in Z(\mathfrak{g})\cap \spann \{a_1,b_2,a_2+b_1\}^\bot$,
\end{enumerate}
where $Z(\mathfrak{g})$ denotes the center of $\mathfrak{g}$.
\end{coro}
\begin{proof}
The statements follow from Theorem~\ref{t3.5} and Lemmas 3 and 4, as well as Corollary 5 in \cite{4}.
\end{proof}
\begin{theorem}\label{t3.8}
Let $G$ be a Lie group with Lie algebra $\mathfrak{g}={\mathfrak{g}}'\bigoplus \mathcal{P}$ and a two-dimensional derived subalgebra.
Suppose $G$ is equipped with a left-invariant Randers metric $F$ of Berwald type defined by a left-invariant Riemannian metric $g$ and a left-invariant vector field $X\neq 0$ satisfying $\|X\|_g<1$.
Then the flag curvatures of the Randers metric $F^c$ on the tangent Lie group $TG$ are as follows:
\begin{fleqn}
\[1) \begin{cases}
\tilde{P}= \spann \left\{e^c_1,e^c_2\right\} \\
\tilde{P}= \spann \left\{e^v_1,e^v_2\right\}
\end{cases}
\to
\begin{aligned}[t]
K^{F^c}\left(\tilde{P},e^c_1\right)
& = K^{F^c}\left(\tilde{P},e^c_2\right) = K^{F^c}\left(\tilde{P},e^v_1\right)=K^{F^c}\left(\tilde{P},e^v_2\right) \\
& =\frac{1}{4}{\left\|a_2+b_1\right\|}^2-\left\langle a_1,b_2\right\rangle,
\end{aligned}
\]
\end{fleqn}
\begin{fleqn}
\[2) \tilde{P}= \spann \left\{e^c_1,e^v_1\right\}  \to K^{F^c}\left(\tilde{P},e^c_1\right)=K^{F^c}\left(\tilde{P},e^v_1\right)=-\frac{3}{4}{\left\|a_1\right\|}^2_g,\]
\end{fleqn}
\begin{fleqn}
\[3) \tilde{P}= \spann \left\{e^c_1,e^v_2\right\}  \to K^{F^c}\left(\tilde{P},e^c_1\right)=K^{F^c}\left(\tilde{P},e^v_2\right)=\frac{1}{4}{\left\|a_2\right\|}^2_g-\left\langle a_1,b_2\right\rangle, \]
\end{fleqn}
\begin{fleqn}
\[4) \tilde{P}= \spann \left\{e^c_2,e^v_1\right\}  \to K^{F^c}\left(\tilde{P},e^c_2\right)=K^{F^c}\left(\tilde{P},e^v_1\right)=\frac{1}{4}{\left\|b_1\right\|}^2_g-\left\langle a_1,b_2\right\rangle, \]
\end{fleqn}
\begin{fleqn}
\[5) \tilde{P}= \spann \left\{e^c_2,e^v_2\right\}  \to K^{F^c}\left(\tilde{P},e^c_2\right)=K^{F^c}\left(\tilde{P},e^v_2\right)=-\frac{3}{4}{\left\|b_2\right\|}^2_g,\]
\end{fleqn}
\begin{fleqn}
\[6)\begin{cases}
\tilde{P}= \spann \left\{e^c_1,U^c\right\} \\
\tilde{P}= \spann \left\{e^v_1,U^c\right\} \end{cases}\!\!
\to
\begin{aligned}[t]
K^{F^c}\left(\tilde{P},e^c_1\right)
&=K^{F^c}\left(\tilde{P},e^v_1\right)
 =\frac{1}{4}\left({\langle b_1,U\rangle }^2-3{\langle a_2,U\rangle }^2+{\|f_1 (U)\|}^2\right)
\\
&-{\left\langle a_1,U\right\rangle }^2-\frac{1}{2}\left\langle a_2,U\right\rangle \left\langle b_1,U\right\rangle,
\end{aligned}
 \]
\end{fleqn}
\begin{fleqn}
\[7) \tilde{P}= \spann \left\{e^c_1,U^v\right\} \to
\begin{aligned}[t]
K^{F^c}\left(\tilde{P},e^c_1\right)=K^{F^c}\left(\tilde{P},U^v\right)
& =-\frac{1}{4}\left(7{\langle a_1,U\rangle }^2+3{\langle a_2,U\rangle }^2\right)\\
&-\langle a_2,U\rangle \langle b_1,U\rangle,
\end{aligned}
\]
\end{fleqn}
\begin{fleqn}
\[8) \tilde{P}= \spann \left\{e^c_2,U^v\right\} \to
\begin{aligned}[t]
 K^{F^c}\left(\tilde{P},e^c_2\right)
& =K^{F^c}\left(\tilde{P},U^v\right)\\
&=-\frac{1}{4}\left(7{\left\langle b_2,U\right\rangle }^2+3{\left\langle b_1,U\right\rangle }^2\right)-\left\langle a_2,U\right\rangle \left\langle b_1,U\right\rangle,
\end{aligned}
\]
\end{fleqn}
\begin{fleqn}
\[9) \tilde{P}= \spann \left\{e^v_1,U^v\right\} \to K^{F^c}\left(\tilde{P},e^v_1\right)=K^{F^c}\left(\tilde{P},U^v\right)=\frac{1}{4}\left({\langle a_1,U\rangle }^2+{\langle b_1,U\rangle }^2+{\|f_1 (U) \|}^2\right),\]
\end{fleqn}
\begin{fleqn}
\[10) \tilde{P}= \spann \left\{e^v_2,U^v\right\}  \to\! K^{F^c}\left(\tilde{P},e^v_2\right) = K^{F^c}\left(\tilde{P},U^v\right) = \frac{1}{4}\left({\langle a_2,U\rangle }^2+{\langle b_2,U\rangle }^2+{\|f_2 (U) \|}^2\right),\]
\end{fleqn}
\begin{fleqn}
\[11)\begin{cases}
\tilde{P}= \spann \left\{e^c_2,U^c\right\} \\
\tilde{P}= \spann \left\{e^v_2,U^c\right\} \end{cases}\!
\!\!\!\!\to\!
\begin{aligned}[t]
 K^{F^c}\left(\tilde{P},e^v_2\right)
 = K^{F^c}\left(\tilde{P},e^c_2\right)
& = \frac{1}{4}\left({\langle a_2,U\rangle }^2-3{\langle b_1,U\rangle }^2+{\|f_2 (U) \|}^2\right) \\
& -\langle b_2,U\rangle^2-\frac{1}{2}\langle a_2,U\rangle \langle b_1,U\rangle,
\end{aligned}
 \]
\end{fleqn}
\begin{fleqn}
\[12)\begin{cases}
\tilde{P}= \spann \left\{U^c,V^c\right\} \\
\tilde{P}= \spann \left\{U^c,V^v\right\} \end{cases}
 \to K^{F^c}\left(\tilde{P},U^c\right)=-\frac{3}{4{\left(1+g\left(X,U\right)\right)}^2}\left({\left\langle f_1 (U) ,V\right\rangle }^2+{\left\langle f_2 (U) ,V\right\rangle }^2\right),\]
\end{fleqn}
\begin{fleqn}
\begin{align*}
& 13)\begin{cases}
\tilde{P}= \spann \left\{U^c,e^c_1\right\} \\
\tilde{P}= \spann \left\{U^c,e^v_1\right\} \end{cases}
\to \\
& \begin{aligned}[t]
 K^{F^c}\!(\tilde{P},U^c)\!=\!\frac{1}{{(1+g(X,U))}^2}
 \left\{\frac{1}{4}\!\left({\langle b_1,U\rangle }^2\!-\!3{\langle a_2,U\rangle }^2\!+{\|f_1 (U) \|}^2\right)\! -\! {\langle a_1,U\rangle }^2\!-\!\frac{1}{2}\langle a_2,U\rangle \left\langle b_1,U\right\rangle \right\},
 \end{aligned}
\end{align*}
\end{fleqn}
\begin{fleqn}
\begin{align*}
& 14)\begin{cases}
\tilde{P}= \spann \left\{U^c,e^c_2\right\} \\
\tilde{P}= \spann \left\{U^c,e^v_2\right\} \end{cases}
\to
\\
&
\begin{aligned}[t]
& K^{F^c}(\tilde{P},U^c)\!=\!\frac{1}{{\left(1+g(X,U)\right)}^2}
 \left\{\frac{1}{4}\!\left({\langle a_2,U\rangle }^2\!-3{\langle b_1,U\rangle }^2\!+{\|f_2 (U)\|}^2\right)\! -\! {\langle b_2,U\rangle }^2\!-\frac{1}{2}\langle a_2,U\rangle \langle b_1,U\rangle \right\},
\end{aligned}
\end{align*}
\end{fleqn}
\begin{fleqn}
\[15)\begin{cases}
\tilde{P}= \spann \left\{U^v,V^v\right\} \\
\tilde{P}= \spann \left\{U^c,U^v\right\}  \end{cases}
 \to K^{F^c}\left(\tilde{P},U^c\right)=K^{F^c}\left(\tilde{P},U^v\right)=0,\]
\end{fleqn}
\begin{fleqn}
\[16) \;
 \tilde{P}= \spann \left\{U^v,V^c\right\} \; \to \; K^{F^c}\left(\tilde{P},U^v\right)=-\frac{3}{4}\left({\left\langle f_1 (U) ,V\right\rangle }^2+{\left\langle f_2 (U) ,V\right\rangle }^2\right),\]
\end{fleqn}
where $\{U,V\}$ is an orthonormal set in $\mathcal{P}$, with respect to the Riemannian metric $g$.
\end{theorem}
\begin{proof}
For instance, we prove (3); the proofs of the other cases are similar. Let $\tilde{P}\!=\!\spann\{e^c_1,e^v_2\}$.
By Theorem~6 in \cite{4}, the first part of Corollary~\ref{t3.7} and the proof of Lemma~\ref{t3.1}, and of course some straightforward computation, we have:
\begin{align*}
K^{F^c}\left(\tilde{P},e^c_1\right)
& =\frac{1}{{\left(1+g(X,e_1\right)}^2}\left\{K\left(e_2,e_1\right)+\frac{1}{2}g\left(\left[e_2, \nabla _{e_1}e_2\right],e_1\right)-\frac{1}{2}g\left( \nabla _{e_2} \ad ^*_{e_2}e_1 , e_1\right)\right.
\\
& \hspace*{31mm} \left. +\frac{1}{4}g\left(\left[e_2, \ad ^*_{e_2}e_1\right], e_1\right)-\frac{1}{2}g\left(\left[\left[e_1,e_2\right],e_2\right], e_1\right)\right\}
\\
& =K\left(e_1,e_2\right)-\frac{1}{2}\left\langle b_1, a_2+b_1\right\rangle +\frac{1}{4}\left\langle b_1,b_1\right\rangle =\frac{1}{4}{\left\|a_2\right\|}^2_g-\left\langle a_1,b_2\right\rangle .
\qedhere
\end{align*}
\end{proof}

\begin{theorem}\label{t3.9}
Let $G$ be a Lie group with Lie algebra $\mathfrak{g}={\mathfrak{g}}'\bigoplus \mathcal{P}$ and a two-dimensional derived subalgebra. Suppose $G$ is equipped with a left-invariant Randers metric $F$ defined by a left-invariant Riemannian metric $g$ and a left-invariant vector field $X\neq 0$ with ${\left\|X\right\|}_g<1$, such that the left-invariant Randers metric $F^v$ is of Berwald type. Then the flag curvatures of $F^v$ on the tangent Lie group $TG$ are as follows:
\begin{fleqn}
\[1) \tilde{P}= \spann \left\{e^c_1,e^c_2\right\} \to
\begin{aligned}[t]
K^{F^v}\left(\tilde{P},e^c_1\right) & =K^{F^v}\left(\tilde{P},e^c_2\right)=K^{F^v}\left(\tilde{P},e^v_1\right)=K^{F^v}\left(\tilde{P},e^v_2\right)\\
& =\frac{1}{4}{\left\|a_2+b_1\right\|}^2-\left\langle a_1,b_2\right\rangle,
\end{aligned}
\]
\end{fleqn}
\begin{fleqn}
\[2)\begin{cases}
\tilde{P}= \spann \left\{e^c_1,U^c\right\} \\
\tilde{P}= \spann \left\{U^c,e^v_1\right\} \end{cases}
\!\!\!\to
\begin{aligned}[t] K^{F^v}\left(\tilde{P},e^c_1\right)
 =K^{F^v}\left(\tilde{P},e^v_1\right)
&=\frac{1}{4}\left({\left\langle b_1,U\right\rangle }^2-3{\left\langle a_2,U\right\rangle }^2+{\left\|f_1 (U) \right\|}^2\right)
\\
&-{\left\langle a_1,U\right\rangle }^2-\frac{1}{2}\left\langle a_2,U\right\rangle \left\langle b_1,U\right\rangle,
\end{aligned}
\]
\end{fleqn}
\begin{fleqn}
\[3)\begin{cases}
\tilde{P}= \spann \left\{e^c_2,U^c\right\} \\
\tilde{P}= \spann \left\{U^c,e^v_2\right\} \end{cases}\!
\!\!\!\! \to \begin{aligned}[t] K^{F^v}\left(\tilde{P},e^c_2\right)=K^{F^v}\left(\tilde{P},U^c\right)
& =\frac{1}{4}\left({\left\langle a_2,U\right\rangle }^2-3{\left\langle b_1,U\right\rangle }^2+{\left\|f_2 (U) \right\|}^2\right)\\
& -{\left\langle b_2,U\right\rangle }^2-\frac{1}{2}\left\langle a_2,U\right\rangle \left\langle b_1,U\right\rangle,
\end{aligned}
 \]
\end{fleqn}
\begin{fleqn}
\[4) \begin{cases}
\tilde{P}= \spann \left\{U^c,V^c\right\} \\
\tilde{P}= \spann \left\{U^c,V^v\right\} \end{cases}
 \to K^{F^v}\left(\tilde{P},U^c\right)=-\frac{3}{4}\left({\left\langle f_1 (U) ,V\right\rangle }^2+{\left\langle f_2 (U) ,V\right\rangle }^2\right),\]
\end{fleqn}
\begin{fleqn}
\[5) \begin{cases}
\tilde{P}= \spann \left\{U^v,V^v\right\} \\
\tilde{P}= \spann \left\{U^v,U^c\right\} \end{cases}
 \to K^{F^v}\left(\tilde{P},U^c\right)=K^{F^v}\left(\tilde{P},U^v\right)=0,\]
\end{fleqn}
\begin{fleqn}
\[6) \; \tilde{P}= \spann \left\{e^c_1,e^v_1\right\} \to K^{F^v}\left(\tilde{P},e^c_1\right)=-\frac{3}{4}{\left\|a_1\right\|}^2_g,\]
\end{fleqn}
\begin{fleqn}
\[7) \; \tilde{P}= \spann \left\{e^c_1,e^v_2\right\} \to K^{F^v}\left(\tilde{P},e^c_1\right)=\frac{1}{4}{\left\|a_2\right\|}^2_g-\left\langle a_1,b_2\right\rangle, \]
\end{fleqn}
\begin{fleqn}
\[8) \; \tilde{P}= \spann \left\{e^c_2,e^v_1\right\} \to K^{F^v}\left(\tilde{P},e^c_2\right)=\frac{1}{4}{\left\|b_1\right\|}^2_g-\left\langle a_1,b_2\right\rangle, \]
\end{fleqn}
\begin{fleqn}
\[9) \; \tilde{P}= \spann \left\{e^c_2,e^v_2\right\} \to K^{F^v}\left(\tilde{P},e^c_2\right)=-\frac{3}{4}{\left\|b_2\right\|}^2_g,\]
\end{fleqn}
\begin{fleqn}
\[10) \; \tilde{P}= \spann \left\{e^v_1,e^c_1\right\} \to K^{F^v}\left(\tilde{P},e^v_1\right)=-\frac{3}{4{\left(1+g\left(X,e_1\right)\right)}^2}{\left\|a_1\right\|}^2_g,\]
\end{fleqn}
\begin{fleqn}
\[11) \; \tilde{P}= \spann \left\{e^v_1,e^c_2\right\} \to K^{F^v}\left(\tilde{P},e^v_1\right)=\frac{1}{{\left(1+g\left(X,e_1\right)\right)}^2}\left(\frac{1}{4}{\left\|b_1\right\|}^2_g-\left\langle a_1,b_2\right\rangle \right),\]
\end{fleqn}
\begin{fleqn}
\[12) \; \tilde{P}= \spann \left\{e^v_1,e^v_2\right\} \to K^{F^v}\left(\tilde{P},e^v_1\right)=\frac{1}{{\left(1+g\left(X,e_1\right)\right)}^2}\left(\frac{1}{4}{\left\|a_2+b_1\right\|}^2-\left\langle a_1,b_2\right\rangle \right),\]
\end{fleqn}
\begin{fleqn}
\[13) \; \tilde{P}= \spann \left\{e^v_2,e^c_1\right\} \to K^{F^v}\left(\tilde{P},e^v_2\right)=\frac{1}{{\left(1+g\left(X,e_2\right)\right)}^2}\left(\frac{1}{4}{\left\|a_2\right\|}^2_g-\left\langle a_1,b_2\right\rangle \right),\]
\end{fleqn}
\begin{fleqn}
\[14) \; \tilde{P}= \spann \left\{e^v_2,e^v_1\right\} \to K^{F^v}\left(\tilde{P},e^v_2\right)=\frac{1}{{\left(1+g\left(X,e_2\right)\right)}^2}\left(\frac{1}{4}{\left\|a_2+b_1\right\|}^2-\left\langle a_1,b_2\right\rangle \right),\]
\end{fleqn}
\begin{fleqn}
\[15) \; \tilde{P}= \spann \left\{e^v_2,e^c_2\right\} \to K^{F^v}\left(\tilde{P},e^v_2\right)=-\frac{3}{4{\left(1+g\left(X,e_2\right)\right)}^2}{\left\|b_2\right\|}^2_g,\]
\end{fleqn}
\begin{fleqn}
\[16) \; \tilde{P}= \spann \left\{e^c_1,U^v\right\} \to K^{F^v}\left(\tilde{P},e^c_1\right)=-\frac{1}{4}\left(7{\left\langle a_1,U\right\rangle }^2+3{\left\langle a_2,U\right\rangle }^2\right)-\left\langle a_2,U\right\rangle \left\langle b_1,U\right\rangle, \]
\end{fleqn}
\begin{fleqn}
\[17) \; \tilde{P}= \spann \left\{e^c_2,U^v\right\} \to K^{F^v}\left(\tilde{P},e^c_2\right)=-\frac{1}{4}\left(7{\left\langle b_2,U\right\rangle }^2+3{\left\langle b_1,U\right\rangle }^2\right)-\left\langle a_2,U\right\rangle \left\langle b_1,U\right\rangle, \]
\end{fleqn}
\begin{fleqn}
\[18) \; \tilde{P}= \spann \left\{e^v_1,U^c\right\} \to \!
\begin{aligned}[t]
 K^{F^c}\left(\tilde{P},e^v_1\right) = \frac{1}{{\left(1+g\left(X,e_1\right)\right)}^2}
 & \left\{\frac{1}{4}\left({\left\langle b_1,U\right\rangle }^2\!-\!3{\left\langle a_2,U\right\rangle }^2+{\left\|f_1 (U) \right\|}^2\right) \right.
\\
&\quad \left. -{\left\langle a_1,U\right\rangle }^2-\frac{1}{2}\left\langle a_2,U\right\rangle \left\langle b_1,U\right\rangle \right\},
 \end{aligned}\]
\end{fleqn}
\begin{fleqn}
\[19) \; \tilde{P}= \spann \left\{e^v_1,U^v\right\} \to K^{F^v}\left(\tilde{P},e^v_1\right)=\frac{1}{4{\left(1+g\left(X,e_1\right)\right)}^2}\left({\left\langle a_1,U\right\rangle }^2+{\left\langle b_1,U\right\rangle }^2+{\left\|f_1 (U) \right\|}^2\right),\]
\end{fleqn}
\begin{fleqn}
\[20) \; \tilde{P}= \spann \left\{e^v_2,U^c\right\} \!\to \!
\begin{aligned}[t]
 K^{F^v}\left(\tilde{P},e^v_2\right)=\frac{1}{{\left(1+g\left(X,e_2\right)\right)}^2}
 & \left\{\frac{1}{4}\left({\left\langle a_2,U\right\rangle }^2-3{\left\langle b_1,U\right\rangle }^2+{\left\|f_2 (U) \right\|}^2\right)\right. \\
 &\quad \left.-{\left\langle b_2,U\right\rangle }^2-\frac{1}{2}\left\langle a_2,U\right\rangle \left\langle b_1,U\right\rangle \right\},
\end{aligned}
 \]
\end{fleqn}
\begin{fleqn}
\[21) \; \tilde{P}= \spann \left\{e^v_2,U^v\right\} \to K^{F^v}\left(\tilde{P},e^v_2\right)=\frac{1}{4{\left(1+g\left(X,e_2\right)\right)}^2}\left({\left\langle a_2,U\right\rangle }^2+{\left\langle b_2,U\right\rangle }^2+{\left\|f_2 (U) \right\|}^2\right),\]
\end{fleqn}
\begin{fleqn}
\[22) \; \tilde{P}= \spann \left\{U^v,V^c\right\} \to K^{F^v}\left(\tilde{P},U^v\right)=-\frac{3}{4{\left(1+g\left(X,U\right)\right)}^2}\left({\left\langle f_1 (U) ,V\right\rangle }^2+{\left\langle f_2 (U) ,V\right\rangle }^2\right),\]
\end{fleqn}
\begin{fleqn}
\[23) \; \tilde{P}= \spann \left\{U^v,e^v_1\right\} \to K^{F^v}\left(\tilde{P},U^v\right)=\frac{1}{4{\left(1+g\left(X,U\right)\right)}^2}\left({\left\langle a_1,U\right\rangle }^2+{\left\langle b_1,U\right\rangle }^2+{\left\|f_1 (U) \right\|}^2\right),\]
\end{fleqn}
\begin{fleqn}
\[24) \; \tilde{P}= \spann \left\{U^v,e^v_2\right\} \to K^{F^v}\left(\tilde{P},U^v\right)=\frac{1}{4{\left(1+g\left(X,U\right)\right)}^2}\left({\left\langle a_2,U\right\rangle }^2+{\left\langle b_2,U\right\rangle }^2+{\left\|f_2 (U) \right\|}^2\right),\]
\end{fleqn}
\begin{fleqn}
\[25) \; \tilde{P}= \spann \left\{U^v,e^c_1\right\} \to
\begin{aligned}[t]
K^{F^v}\left(\tilde{P},U^v\right)=-\frac{1}{{\left(1+g\left(X,U\right)\right)}^2}
& \left\{\frac{1}{4}\left(7{\left\langle a_1,U\right\rangle }^2+3{\left\langle a_2,U\right\rangle }^2\right) \right.
\\
& \quad +\left\langle a_2,U\right\rangle \left\langle b_1,U\right\rangle \bigg\},
 \end{aligned}
 \]
\end{fleqn}
\begin{fleqn}
\[26) \; \tilde{P}= \spann \left\{U^v,e^c_2\right\} \to
\begin{aligned}[t]
 K^{F^v}\left(\tilde{P},U^v\right)=-\frac{1}{{\left(1+g\left(X,U\right)\right)}^2}
& \left\{\frac{1}{4}\left(7{\left\langle b_2,U\right\rangle }^2+3{\left\langle b_1,U\right\rangle }^2\right)\right. \\
&\quad +\left\langle a_2,U\right\rangle \left\langle b_1,U\right\rangle \bigg\},
\end{aligned}
 \]
\end{fleqn}
where $\{U,V\}$ is an orthonormal set in $\mathcal{P}$, with respect to the Riemannian metric $g$.
\end{theorem}
\begin{proof}
The proof is available by referring to Theorem 7 in \cite{4} and analogous computations to the proof of the previous theorem.
\end{proof}

Hence, taking into account the assumptions of the previous two theorems and the structure of the Lie algebra described, we have:
\begin{coro}\label{t3.10}
The tangent bundle of a Lie group with a two-dimensional derived subalgebra has positive, negative, and zero flag curvatures at every point.
\end{coro}
Finally, we provide an example of a family of Lie groups with two-dimensional derived subalgebras and compute some geometric quantities on the tangent Lie group.

\begin{example}\label{t3.11}
Let $G$ be a five-dimensional Lie group with Lie algebra $\mathfrak{g}=\spann\left\{e_1,e_2,Y_1,Y_2,Y_3\right\}$ defined by the following non-zero brackets:
\[\left[Y_1,Y_2\right]=e_1,  \ \ \  \left[Y_1,e_1\right]=e_2.\]
Also, Let $g=\left\langle \cdot ,
\cdot \right\rangle $ be the left-invariant Riemannian metric on it such that $\left\{e_1,e_2,Y_1,Y_2,Y_3\right\}$ is an orthonormal basis. So we have ${\mathfrak{g}}'= \spann \left\{e_1,e_2\right\}$ and $\mathcal{P}= \spann \left\{Y_1,Y_2,Y_3\right\}$. Easily we can see that $a_2=Y_1$, $a_1=b_1=b_2=0$ and the matrix representations of $f_1$ and $f_2$ with respect to the basis $\left\{Y_1,Y_2,Y_3\right\}$ are as follows:
\[f_1=\begin{pmatrix}
0 & -1 & 0 \\
1 & 0 & 0 \\
0 & 0 & 0 \end{pmatrix}
, \qquad f_2= \begin{pmatrix}
0 & 0 & 0 \\
0 & 0 & 0 \\
0 & 0 & 0 \end{pmatrix}.
\]
Now, using the complete and vertical lifts, we construct a natural Riemannian metric $\tilde{g}$ on $TG$.
Then, by Theorems~\ref{t3.2} and \ref{t3.4}, the sectional and Ricci curvatures in various cases are obtained as follows:
\begin{fleqn}
\begin{align*}
\tilde{K}\left(e^c_1,e^c_2\right)
& =\tilde{K}\left(e^v_1,e^v_2\right)= \tilde{K}\left(e^c_1,e^v_2\right)
=\tilde{K}\left(Y^c_2,e^c_1\right)
=\tilde{K}\left(Y^c_2,e^v_1\right)
=\tilde{K}\left(Y^c_1,e^c_2\right)
=\tilde{K}\left(Y^c_1,e^v_2\right)
\\
&
=\tilde{K}\left(Y^v_1,e^v_1\right)
=\tilde{K}\left(Y^v_2,e^v_1\right)
=\tilde{K}\left(Y^v_1,e^v_2\right)=\frac{1}{4},
\end{align*}
\begin{align*}
&
\tilde{K}\left(Y^c_1,e^c_1\right)=\tilde{K}\left(Y^c_1,e^v_1\right)=-\frac{1}{2} , \qquad \tilde{K}\left(Y^c_1,Y^c_2\right)
=\tilde{K}\left(Y^c_1,Y^v_2\right)
=\tilde{K}\left(Y^v_1,e^c_1\right)=-\frac{3}{4},
\end{align*}
\begin{align*}
\tilde{K}\left(e^c_1,e^v_1\right)
& =\tilde{K}\left(e^v_1,e^c_2\right)
=\tilde{K}\left(e^c_2,e^v_2\right)
=\tilde{K}\left(Y^c_1,Y^c_3\right)
=\tilde{K}\left(Y^c_1,Y^v_3\right)
=\tilde{K}\left(Y^c_2,Y^c_3\right)
=\tilde{K}\left(Y^c_2,Y^v_3\right)
\\
& =\tilde{K}\left(Y^c_3,e^c_1\right)
=\tilde{K}\left(Y^c_3,e^v_1\right)
=\tilde{K}\left(Y^c_2,e^c_2\right)
=\tilde{K}\left(Y^c_2,e^v_2\right)
=\tilde{K}\left(Y^c_3,e^c_2\right)
=\tilde{K}\left(Y^c_3,e^v_2\right)
\\
& =\tilde{K}\left(Y^v_2,e^c_1\right)
=\tilde{K}\left(Y^v_3,e^c_1\right)
=\tilde{K}\left(Y^v_3,e^v_1\right)
=\tilde{K}\left(Y^v_2,e^v_2\right)
=\tilde{K}\left(Y^v_3,e^v_2\right)=0,
\end{align*}
\[\tilde{K}\left(u^v,v^v\right)=\tilde{K}\left(u^c,u^v\right)=\tilde{K}\left(u^v,e^c_2\right)=0, \qquad \forall u,v\in \left\{Y_1,Y_2,Y_3\right\}.\]
\end{fleqn}
Also, considering that,
\[tr\left(f^2_1\right)=-1 , \qquad tr\left(f^2_2\right)=tr\left(f_1of_2\right)=tr\left(f_2of_1\right)=0,\]
and some computation, we obtain:
\begin{align*}
& \tilde{r}\left(e^c_1\right)=-\frac{3}{4} , \quad \tilde{r}\left(e^c_2\right)=\frac{1}{2} , \quad \tilde{r}\left(e^v_1\right)=0 , \quad \tilde{r}\left(e^v_2\right)=1 , \quad \tilde{r}\left(Y^c_1\right)=-2 ,\\
& \tilde{r}\left(Y^c_1\right)=-2 , \quad \tilde{r}\left(Y^c_2\right)=-1 , \quad \tilde{r}\left(Y^c_3\right)=0 , \quad \tilde{r}\left(Y^v_1\right)=-1 , \quad \tilde{r}\left(Y^v_2\right)=-\frac{1}{2} ,\\
& \tilde{r}\left(Y^v_3\right)=0 , \quad \widetilde{\Ric }\left(e^c_1,e^2_2\right)
=\widetilde{\Ric }\left(e^v_1,e^v_2\right)=0.
\end{align*}
The remaining Ricci curvatures are also equal to zero.

If $G$ is equipped with a left-invariant Berwald-type Randers metric $F$ defined by $g$ and a left-invariant vector field $X\neq 0$ with $\|X\|_g<1$, then for instance, taking the transverse vector $U=Y_1$ we obtain
\begin{align*}
& \begin{cases}
\tilde{P}= \spann \left\{e^c_1,Y^v_1\right\} \\
\tilde{P}= \spann \left\{Y^v_1,e^c_1\right\} \end{cases}
\to K^{F^c}\left(\tilde{P},e^c_1\right)=K^{F^c}\left(\tilde{P},Y^v_1\right)=-\frac{3}{4},
\\[1ex]
& \begin{cases}
\tilde{P}= \spann \left\{e^c_2,Y^c_1\right\} \\
\tilde{P}= \spann \left\{e^v_2,Y^c_1\right\} \end{cases}
 \to K^{F^c}\left(\tilde{P},e^v_2\right)=K^{F^c}\left(\tilde{P},e^c_2\right)=\frac{1}{4}.
\end{align*}
Similarly, the other flag curvatures with respect to $F^c$ are obtained.

Now, assuming that $G$ is equipped with a left-invariant Randers metric $F$ defined by a left-invariant Riemannian metric $g$ and a left-invariant vector field $X\neq 0$ with the condition $\|X\|_g<1$ such that the left-invariant Randers metric $F^v$ is of Brewald type, then for the flag curvature with respect to $F^v$ on $TG$, taking $U=Y_1$ as the flagpole and $V=Y_2$ as the transverse vector, we have:
\[\tilde{P}= \spann \left\{Y^v_1,Y^c_2\right\} \longrightarrow K^{F^v}\left(\tilde{P},Y^v_1\right)=-\frac{3}{4{\left(1+g\left(X,Y_1\right)\right)}^2}.\]
Computations for the remaining cases are similar.
\end{example}

\section{\bfseries Addendum}
\label{s4}
This section presents further results and computations for a Lie group
$G$ whose commutator subgroup has dimension two, endowed with a left-invariant Riemannian metric $g$, as well as for its tangent Lie group.

\begin{remark}\label{t4.1}
We mention that a vector field $U$ on a Lie group $G$ is a geodesic vector field if $ \nabla _UU=0$ at all points of $G$. It is clear that for any $X\in \mathcal{P}$, we have $ \nabla _XX=0$. Hence, every $X\in \mathcal{P}$ is a geodesic vector. If $X=\lambda e_1+\mu e_2\in{\mathfrak{g}}'$, then $X$ is a geodesic vector if and only if, $\lambda^2a_1+\lambda \mu (a_2+b_1)+\mu^2b_2=0$.
\end{remark}
\begin{remark}
Easily we can see that if $X\in{\mathfrak{g}}'$ then, $\tr\ad_X=0$. Also if $X\in\mathcal{P}$, then $\tr\ad_X=a_1+b_2$. Hence, $\tr\ad_X=0$ if and only if $a_1=-b_2$. We recall that a Lie algebra $\mathfrak{g}$ is called unimodular if $\tr\ad_X=0$ for every $X\in\mathfrak{g}$. Therefore, a Lie group $G$ with a two-dimensional derived subalgebra ${\mathfrak{g}}'$ is unimodular if and only if $a_1=-b_2$. As a special case, if $G$ is equipped with a left-invariant Riemannian metric such that $e_1$, $e_2$ are geodesic vectors, then $G$ is unimodular.
\end{remark}
\begin{prop}\label{t4.3}
A Lie group $G$ with a two-dimensional commutator subgroup does not admit a bi-invariant Riemannian metric.
\end{prop}
\begin{proof}
A left-invariant Riemannian metric $g=\left\langle \cdot , \cdot \right\rangle $ on a connected Lie group $G$ is bi-invariant if and only if
\[\left\langle X, [Y,Z] \right\rangle =\left\langle [X,Y] ,Z\right\rangle \qquad \forall X,Y,Z\in \mathfrak{g}.\]
Suppose $\left\langle \cdot , \cdot \right\rangle $ is a bi-invariant Riemannian metric on $G$. For any $u,v\in \mathcal{P}$ and with straightforward computations, we obtain:
\begin{align*}
& \left\langle [u,v] ,e_1\right\rangle =\left\langle u,\left[v,e_1\right]\right\rangle \Longleftrightarrow \left\langle f_1 (u) ,v\right\rangle =0,\\
& \left\langle [u,v] ,e_2\right\rangle =\left\langle u,\left[v,e_2\right]\right\rangle \Longleftrightarrow \left\langle f_2 (u) ,v\right\rangle =0, \\
& \left\langle \left[e_1,e_1\right],u\right\rangle =\left\langle e_1,\left[e_1,u\right]\right\rangle \Longleftrightarrow -\left\langle a_1,u\right\rangle =0\Longleftrightarrow a_1=0, \\
& \left\langle \left[e_2,e_2\right],u\right\rangle =\left\langle e_2,[e_2,u]\right\rangle \Longleftrightarrow -\left\langle b_2,u\right\rangle =0\Longleftrightarrow b_2=0, \\
& \left\langle e_1,\left[u,e_2\right]\right\rangle =\left\langle \left[e_1,u\right],e_2\right\rangle \Longleftrightarrow b_1=-a_2, \\
& \left\langle \left[e_1,e_2\right],u\right\rangle =\left\langle e_1,\left[e_2,u\right]\right\rangle \Longleftrightarrow -\left\langle b_1,u\right\rangle =0\Longleftrightarrow b_1=0, \\
& \left\langle \left[e_2,e_1\right],u\right\rangle =\left\langle e_2,\left[e_1,u\right]\right\rangle \Longleftrightarrow -\left\langle a_2,u\right\rangle =0\Longleftrightarrow a_2=0.
\end{align*}
But this implicitly means that $G$ is an abelian Lie group, which is a contradiction.
\end{proof}

At the end of this section, we present the Riemann curvature tensor on the tangent Lie group $TG$, which is used in the process of computing the sectional curvatures in Theorem~\ref{t3.2} as well.
\begin{lemma}\label{t4.4}
Under the assumptions of Lemma~\ref{t3.1}, the Riemann curvature tensor $\tilde{R}$ on the Riemannian manifold $(TG,\tilde{g})$ is given by:
\begin{fleqn}
\[1) \; \tilde{R}\left(e^c_1,e^c_2\right)e^c_2=\left(-\left\langle a_1,b_2\right\rangle +\frac{1}{4}{\left\|a_2+b_1\right\|}^2\right)e^c_1-\frac{1}{4}{\left(2f_1\left(b_2\right)-f_2\left(a_2\right)-f_2\left(b_1\right)\right)}^c,\]
\[2)\;\tilde{R}\left(e^c_1,e^v_2\right)e^v_2=\left(-\left\langle a_1,b_2\right\rangle +\frac{1}{4}{\left\|a_2\right\|}^2\right)e^c_1-\frac{1}{4}\left\langle 2b_1+a_2,b_2\right\rangle e^c_2-\frac{1}{4}{\left(2f_1\left(b_2\right)-f_2\left(a_2\right)\right)}^c,\]
\[3)\;\tilde{R}\left(e^c_1,e^v_1\right)e^v_1=-\frac{1}{4}\left(3{\left\|a_1\right\|}^2e^c_1+\left\langle b_1+2a_2,a_1\right\rangle e^c_2+{\left(f_1\left(a_1\right)\right)}^c\right), \]
\[4)\;\tilde{R}\left(e^v_1,e^v_2\right)e^v_2=\left(-\left\langle a_1,b_2\right\rangle +\frac{1}{4}{\left\|a_2+b_1\right\|}^2\right)e^v_1-\frac{1}{4}{\left(2f_1\left(b_2\right)-f_2\left(a_2\right)-f_2\left(b_1\right)\right)}^v, \]
\[5)\;\tilde{R}\left(e^v_1,e^c_2\right)e^c_2=\left(-\left\langle a_1,b_2\right\rangle +\frac{1}{4}{\left\|b_1\right\|}^2\right)e^v_1-\frac{1}{4}\left\langle b_1+2a_2,b_2\right\rangle e^v_2-\frac{1}{2}{\left(f_1\left(b_2\right)\right)}^v, \]
\[6)\;\tilde{R}\left(e^c_2,e^v_2\right)e^v_2=-\frac{1}{4}\left(\left\langle 2b_1+a_2,b_2\right\rangle e^c_1+3{\left\|b_2\right\|}^2e^c_2+{\left(f_2\left(b_2\right)\right)}^c\right), \]
\[7)\;\tilde{R}\left(e^c_2,e^v_1\right)e^v_1=-\frac{1}{4}\left\langle a_1,b_1+2a_2\right\rangle e^c_1+\left(\frac{1}{4}{\left\|b_1\right\|}^2-\left\langle a_1,b_2\right\rangle \right)e^c_2+{\left(\frac{1}{4}f_1\left(b_1\right)-\frac{1}{2}f_2\left(a_1\right)\right)}^c, \]
\[8)\;\tilde{R}\left(e^c_2,e^v_2\right)e^v_2=-\frac{1}{4}\left\langle b_2,a_2+2b_1\right\rangle e^c_1-\frac{3}{4}{\left\|b_2\right\|}^2e^c_2-\frac{1}{4}{\left(f_2\left(b_2\right)\right)}^c, \]
\[9)\;\tilde{R}\left(e^v_2,e^c_2\right)e^c_2=-\frac{1}{4}\left\langle {2a}_2+b_1,b_2\right\rangle e^v_1-\frac{3}{4}{\left\|b_2\right\|}^2e^v_2-\frac{1}{2}{\left(f_2\left(b_2\right)\right)}^v, \]
\[10)\;\tilde{R}\left(e^v_2,e^v_1\right)e^v_1=\left(\frac{1}{4}{\left\|a_2+b_1\right\|}^2-\left\langle a_1,b_2\right\rangle \right)e^v_2+\frac{1}{4}{\left(f_1\left(a_2\right)\right)}^v+\frac{1}{4}{\left(f_1\left(b_1\right)\right)}^v-\frac{1}{2}{\left(f_2\left(a_1\right)\right)}^v, \]
\[11)\;\tilde{R}\left(e^v_2,e^c_1\right)e^c_1=-\frac{1}{4}\left\langle a_1,a_2+2b_1\right\rangle e^v_1+\left(\frac{1}{4}{\left\|a_2\right\|}^2-\left\langle a_1,b_2\right\rangle \right)e^v_2-\frac{1}{2}{\left(f_2\left(a_1\right)\right)}^v, \]
\[12)\;\tilde{R}\left(e^c_2,e^c_1\right)e^c_1=\left(-\left\langle a_1,b_2\right\rangle +\frac{1}{4}{\left\|a_2+b_1\right\|}^2\right)e^c_2-\frac{1}{4}{\left(2f_2\left(a_1\right)-f_1\left(a_2\right)-f_1\left(b_1\right)\right)}^c, \]
\[
13)\;
\begin{aligned}[t]
 \tilde{R}\left(u^c,v^c\right)v^c
 & =\left(\frac{1}{4}\left\langle f_2 (u) ,v\right\rangle \left\langle b_1+3a_2,v\right\rangle +\left\langle f_1 (u) ,v\right\rangle \left\langle a_1,v\right\rangle \right)e^c_1
 \\
& +\left(\frac{1}{4}\left\langle f_1 (u) ,v\right\rangle \left\langle a_2+3b_1,v\right\rangle +\left\langle f_2 (u) ,v\right\rangle \left\langle b_2,v\right\rangle \right)e^c_2
\\
& +\frac{3}{4}{\Big(\left\langle f_1 (u) ,v\right\rangle f_1\left(v\right)+\left\langle f_2 (u) ,v\right\rangle f_2(v)\Big)}^c,
\end{aligned}
\]
\[14)\;
\begin{aligned}[t]
\tilde{R}\left(u^c,v^v\right)v^v
& =\frac{3}{4}\left(\left\langle f_1 (u) ,v\right\rangle \left\langle a_1,v\right\rangle +\left\langle f_2 (u) ,v\right\rangle \left\langle a_2,v\right\rangle \right)e^c_1
\\
& +\frac{3}{4}\left(\left\langle f_1 (u) ,v\right\rangle \left\langle b_1,v\right\rangle +\left\langle f_2 (u) ,v\right\rangle \left\langle b_2,v\right\rangle \right)e^c_2
\\
& +\frac{3}{4}{\left(\left\langle f_1 (u) ,v\right\rangle f_1\left(v\right)+\left\langle f_2 (u) ,v\right\rangle f_2\left(v\right)\right)}^c,
\end{aligned}
\]
\[15)\;
\begin{aligned}[t]
\tilde{R}\left(u^v,v^c\right)v^c
& =\left(-\frac{1}{4}\left\langle f_2 (u) ,v\right\rangle \left\langle b_1-a_2,u\right\rangle +\left\langle f_1 (u) ,v\right\rangle \left\langle a_1,v\right\rangle +\frac{1}{2}\left\langle f_2 (u) ,v\right\rangle \left\langle b_1+a_2,u\right\rangle \right)e^v_1
\\
& +\left(-\frac{1}{4}\left\langle f_1 (u) ,v\right\rangle \left\langle a_2-b_1,v\right\rangle +\left\langle f_2 (u) ,v\right\rangle \left\langle b_2,v\right\rangle +\frac{1}{2}\left\langle f_1 (u) ,v\right\rangle \left\langle b_1+a_2,v\right\rangle \right)e^v_2
\\
& +\frac{3}{4}\left\langle f_1 (u) ,v\right\rangle {\left(f_1\left(v\right)\right)}^v+\frac{3}{4}\left\langle f_2 (u) ,v\right\rangle {\left(f_2\left(v\right)\right)}^v,
\end{aligned}
\]

\[16\;)\tilde{R}\left(u^v,v^v\right)v^v=\tilde{R}\left(u^c,u^v\right)u^v=0,\]
\[17)\;
\begin{aligned}[t]
& \tilde{R}\left(u^c,e^c_1\right)e^c_1=
\\
& \qquad \;\;{\left(-\frac{1}{4}f^2_1 (u) -\left\langle a_1,u\right\rangle a_1+\frac{1}{4}\left\langle u,f_1\left(a_2+b_1\right)-2f_2\left(a_1\right)\right\rangle e_2+\frac{1}{4}\left\langle u,b_1-3a_2\right\rangle \left(a_2+b_1\right)\right)}^c,
\end{aligned}
\]
\[18)\;
\begin{aligned}[t]
& \tilde{R}\left(u^c,e^v_1\right)e^v_1=
\\
& {\left(\!-\frac{1}{4}\langle u,f_1(a_1)\rangle e_1\!+\!\frac{1}{4}\langle u,f_1(b_1)\!-2f_2(a_1)\rangle e_2\!+\!\frac{1}{4}\langle u,b_1-3a_2\rangle (a_2+b_1)-\langle a_1,u\rangle a_1\!-\!\frac{1}{4}f^2_1 (u)\! \right)}^c,
\end{aligned}
\]
\[19)\;\tilde{R}(u^c,e^c_2)e^c_2={\left(\!-\frac{1}{4}f^2_2 (u) \!-\! \langle b_2,u \rangle b_2\!+\!\frac{1}{4}\langle u,f_2(b_1+a_2)\!-\!2f_1(b_2)\rangle e_1\!+\frac{1}{4}\langle u,a_2\!-\!3b_1\rangle (a_2\!+b_1)\!\right)}^c,
\]
\[20)\;
\begin{aligned}[t]
&\tilde{R}\left(u^c,e^v_2\right)e^v_2=
\\
&{\left(\!-\frac{1}{4}f^2_2 (u) \!-\langle b_2,u\rangle b_2\!+\frac{1}{4}\langle u,f_2(a_2)-2f_1(b_2)\rangle e_1\!-\frac{1}{4}\langle u,f_2(b_2)\rangle e_2\!+\frac{1}{4}\langle u,a_2\!-3b_1\rangle (a_2+b_1)\!\right)}^c,
\end{aligned}
\]
\[21)\; \tilde{R}\left(u^v,e^c_1\right)e^c_1={\left(-\frac{1}{2}\left(\left\langle u,f_1\left(a_1\right)\right\rangle e_1+\left\langle u,f_2\left(a_1\right)\right\rangle e_2\right)-\frac{3}{4}\left(\left\langle a_1,u\right\rangle a_1+\left\langle a_2,u\right\rangle a_2\right)\right)}^v,\]
\[22)\; \tilde{R}\left(u^v,e^v_1\right)e^v_1={\frac{1}{4}\left(\left\langle u,f_1\left(a_2+b_1\right)-2f_2\left(a_1\right)\right\rangle e_2+\left\langle b_1,u\right\rangle b_1+\left\langle a_1,u\right\rangle a_1-f^2_1 (u) \right)}^v,\]
\[23)\; \tilde{R}\left(u^v,e^c_2\right)e^c_2={\left(-\frac{1}{2}\left(\left\langle u,f_1\left(b_2\right)\right\rangle e_1+\left\langle u,f_2\left(b_2\right)\right\rangle e_2\right)-\frac{3}{4}\left(\left\langle b_1,u\right\rangle b_1+\left\langle b_2,u\right\rangle b_2\right)\right)}^v, \]
\[24)\;\tilde{R}\left(u^v,e^v_2\right)e^v_2=\frac{1}{4}{\left(\left\langle u,f_2\left(a_2+b_1\right)-2f_1\left(b_2\right)\right\rangle e_1+\left\langle a_2,u\right\rangle a_2+\left\langle b_2,u\right\rangle b_2-f^2_2 (u) \right)}^v, \]
\[25)\;
\begin{aligned}[t]
\tilde{R}\left(e^c_2,u^v\right)u^v
& =-\frac{3}{4}\left(\left(\left\langle a_1,u\right\rangle \left\langle b_1,u\right\rangle +\left\langle a_2,u\right\rangle \left\langle b_2,u\right\rangle \right)e^c_1+\left({\left\langle b_1,u\right\rangle }^2+{\left\langle b_2,u\right\rangle }^2\right)e^c_2
\right.\\
& +\left\langle b_1,u\right\rangle {\left(f_1 (u) \right)}^c+\left\langle b_2,u\right\rangle {\left(f_2 (u) \right)}^c\Big),
\end{aligned}
\]
\[26)\;
\begin{aligned}[t]
\tilde{R}\left(e^c_2,u^c\right)u^c
\!&=\!\left(\!\frac{1}{2}\langle b_2,u\rangle \langle b_1\!-\!a_2,u\rangle \!+\!\frac{1}{4}\langle f_1 (u) ,f_2 (u) \rangle\!
-\!\langle a_1,u\rangle \langle b_1,u\rangle \!-\!\frac{1}{2}\langle b_2,u\rangle \langle b_1\!+\!a_2,u\rangle \!\right)e^c_1
\\
& +\!\left(\!-\frac{1}{2}\langle b_1,u\rangle \langle b_1+a_2,u\rangle\! -{\langle b_2,u\rangle }^2\!+\frac{1}{4}\left\langle b_1+a_2,u\right\rangle \left\langle a_2-b_1,u\right\rangle +\frac{1}{4}{\left\|f_2(u)\right\|}^2\right)e^c_2
\\
& -\frac{1}{4}\left\langle a_2+3b_1,u\right\rangle {\left(f_1 (u) \right)}^c-\left\langle b_2,u\right\rangle {\left(f_2 (u) \right)}^c,
\end{aligned}
\]
\[27)\;
\begin{aligned}[t]
\tilde{R}\left(e^v_2,u^v\right)u^v
&=\frac{1}{4}\left(\left\langle a_1,u\right\rangle \left\langle a_2,u\right\rangle +\left\langle b_1,u\right\rangle \left\langle b_2,u\right\rangle +\left\langle f_1 (u) ,f_2(u)\right\rangle \right)e^v_1\\
& +\frac{1}{4}\left({\left\langle a_2,u\right\rangle }^2+{\left\langle b_2,u\right\rangle }^2+{\left\|f_2 (u) \right\|}^2\right)e^v_2,
\end{aligned}
\]
\[28)\;
\begin{aligned}[t]
\tilde{R}\left(e^v_2,u^c\right)u^c
&\!=\!\left(\!\frac{1}{2}\langle b_2,u\rangle \langle b_1-a_2,u\rangle\! +\!\frac{1}{4}\langle f_1 (u) ,f_2 (u) \rangle \!-\!\langle a_1,u\rangle \langle b_1,u\rangle\! -\!\frac{1}{2}\langle b_2,u\rangle \langle b_1+a_2,u\rangle\! \right)\!e^v_1
\\
& +\left(\!-\frac{1}{2}\langle b_1,u\rangle \langle b_1+a_2,u\rangle \!-\!{\langle b_2,u\rangle }^2
\!+\!\frac{1}{4}\langle b_1+a_2,u\rangle \langle a_2-b_1,u\rangle \!+\!\frac{1}{4}{\|f_2 (u) \|}^2\right)e^v_2
\\
& -\frac{1}{4}\left\langle a_2+3b_1,u\right\rangle {\left(f_1 (u) \right)}^v-\left\langle b_2,u\right\rangle {\left(f_2 (u) \right)}^v,
\end{aligned}
\]
\[29)\;
\begin{aligned}[t]
\tilde{R}\left(e^v_1,u^v\right)u^v
&=\frac{1}{4}\left({\left\langle a_1,u\right\rangle }^2+{\left\langle b_1,u\right\rangle }^2+{\left\|f_1 (u) \right\|}^2\right)e^v_1
\\
& +\frac{1}{4}\Big(\left\langle a_1,u\right\rangle \left\langle a_2,u\right\rangle +\left\langle b_1,u\right\rangle \left\langle b_2,u\right\rangle +\left\langle f_1 (u) ,f_2 (u) \right\rangle \Big)e^v_2,
\end{aligned}
\]
\[30)\;
\begin{aligned}[t]
\tilde{R}(e^v_1,u^c)u^c
& =\left(\!-\frac{1}{2}\langle a_2,u\rangle \langle b_1+a_2,u\rangle -{\langle a_1,u\rangle }^2\!+\!\frac{1}{4}\langle b_1+a_2,u\rangle \langle b_1-a_2,u\rangle\! +\!\frac{1}{4}{\|f_1 (u)\|}^2\right)e^v_1
\\
& +\!\left(\!\frac{1}{2}\langle a_1,u\rangle \langle a_2\!-\!b_1,u\rangle \!+\!\frac{1}{4}\langle f_1 (u) ,f_2 (u) \rangle \!-\!\langle a_2,u\rangle \langle b_2,u\rangle\! -\!\frac{1}{2}\langle a_1,u\rangle \langle b_1+a_2,u\rangle \right)e^v_2
\\
& -\frac{1}{4}\langle b_1+3a_2,u\rangle {(f_2 (u))}^v-\langle a_1,u\rangle {(f_1 (u))}^v,
\end{aligned}
\]
\end{fleqn}
\end{lemma}
\begin{proof}
For instance, we prove one case. The proofs of the remaining statements are similar. By the definition of the Riemann curvature tensor,
\begin{align*}
\tilde{R}\left(e^c_1,e^v_1\right)e^v_1
& ={\widetilde \nabla }_{e^c_1}{\widetilde \nabla }_{e^v_1}e^v_1-{\widetilde \nabla }_{e^v_1}{\widetilde \nabla }_{e^c_1}e^v_1-{\widetilde \nabla }_{[e^c_1,e^v_1]}e^v_1
\\
&={\widetilde \nabla }_{e^c_1}a^c_1-\frac{1}{2}{\widetilde \nabla }_{e^v_1}a^v_1=-\left\langle a_1,a_1\right\rangle e^c_1-\frac{1}{2}\left\langle a_2+b_1,a_1\right\rangle e^c_2-\frac{1}{2}{\left(f_1\left(a_1\right)\right)}^c
\\
& +\frac{1}{4}\left(\left\langle a_1,a_1\right\rangle e^c_1+\left\langle a_1,b_1\right\rangle e^c_2+{\left(f_1\left(a_1\right)\right)}^c\right)
\\
&
=-\frac{3}{4}{\left\|a_1\right\|}^2e^c_1-\frac{1}{4}\left\langle a_1,b_1+2a_2\right\rangle e^c_2-\frac{1}{4}\left(f_1{\left(a_1\right)}^c\right)\\
&=-\frac{1}{4}\left(3{\left\|a_1\right\|}^2e^c_1+\left\langle a_1,b_1+2a_2\right\rangle e^c_2+{\left(f_1\left(a_1\right)\right)}^c\right). \qedhere
\end{align*}
\end{proof}

{\large{\textbf{Acknowledgment.}}} We are grateful to the office of Graduate Studies of the University of Isfahan for their support.
\section*{Declarations}

{\textbf{Ethical Approval:} Not applicable.\\

{\textbf{Conflict of interest:} On behalf of all authors, the corresponding author states that there is no conflict of interest.\\

{\textbf{Authors contributions:}} Not applicable.\\

{\textbf{Funding:}} There is not any financial support. \\

{\textbf{Availability of data and materials:}} Data sharing does not apply to this article as no datasets were generated or analyzed during the current study.\\

\end{document}